\documentclass[11pt]{amsart}%
\usepackage{amssymb}
\usepackage{amsmath}
\usepackage{amsfonts}
\usepackage{graphicx}
\theoremstyle{plain}

\newtheorem{definition}{Definition}

\newtheorem{lemma}{Lemma}

\newtheorem{proposition}{Proposition}
\newtheorem{remark}{Remark}

\newtheorem{theorem}{Theorem}
\newcommand{\N}{\mathbb{N}}

\newcommand{\R}{\mathbb{R}}
\newcommand{\C}{\mathbb{C}}
\numberwithin{equation}{section}
\begin{document}
\title[Real analytic approximation of Lipschitz functions]{Real analytic approximation of Lipschitz functions
on Hilbert space and other Banach spaces}
\author{}
\author{D. Azagra}
\address{ICMAT (CSIC-UAM-UC3-UCM), Departamento de An{\'a}lisis Matem{\'a}tico,
Facultad Ciencias Matem{\'a}ticas, Universidad Complutense, 28040,
Madrid, Spain}
\email{daniel\_azagra@mat.ucm.es}
\author{R. Fry}
\address{Department of Mathematics and Statistics, Thompson Rivers University,
Kamloops, B.C., Canada}
\email{Rfry@tru.ca }
\author{L. Keener}
\address{Department of Mathematics and Statistics, University of Northern British
Columbia, Prince George, B.C., Canada}
\email{keener@unbc.ca}
\keywords{Real analytic, approximation, Lipschitz function, Banach
space. The first named author was partly supported by grant
Santander-Complutense 34/07-15844. The second named author was
partly supported by NSERC (Canada).}

\begin{abstract}
Let $X$ be a separable Banach space with a separating polynomial.
We show that there exists $C\geq 1$ (depending only on $X$) such
that for every Lipschitz function $f:X\rightarrow\mathbb{R}$, and
every $\varepsilon>0$, there exists a Lipschitz, real analytic
function $g:X\rightarrow\mathbb{R}$ such that $|f(x)-g(x)|\leq
\varepsilon$ and $\textrm{Lip}(g)\leq  C\textrm{Lip}(f)$. This
result is new even in the case when $X$ is a Hilbert space.
Furthermore, in the Hilbertian case we also show that $C$ can be
assumed to be any number greater than $1$.
\end{abstract}

\maketitle

\section{Introduction and main results}

Not much is known about the natural question of approximating
functions by real analytic functions on a real Banach space $X$.
When $X$ is finite dimensional, a famous paper of Whitney's
\cite{W} provides a completely satisfactory answer to this
problem: a combination of integral convolutions with Gaussian
kernels and real analytic approximations of partitions of unity
allows to show that for every $C^k$ function $f:\R^n\to\R^m$ and
every continuous $\varepsilon:\R^n\to (0, +\infty)$ there exists a
real analytic function $g$ such that $\|D^{j}g(x)-D^{j}f(x)\|\leq
\varepsilon(x)$ for all $x\in\R^n$ and $j=1, ..., k$.

In an infinite dimensional Banach space $X$, the lack of a
translation invariant measure which assigns finite, strictly
positive volume to balls makes this question much more difficult
to answer, as integral convolutions cannot be directly used in
order to regularize a given function. By constructing a real
analytic approximation of a partition of unity, Kurzweil was able
to show in his classic paper \cite{K} that for every Banach space
$X$ having a separating polynomial, for every Banach space $Y$,
and for every continuous function $f:X\to Y$, $\varepsilon:X\to
(0, +\infty)$, there exists a real analytic function $g:X\to Y$
such that $\|f(x)-g(x)\|\leq\varepsilon(x)$ for all $x\in X$.

In the special case of norms (or more generally convex funtions)
defined on a Banach space $X$, Deville, Fonf and Hajek \cite{DFH1,
DFH2}, by introducing specific formulae suited to this problem,
have proved two important results: first, in $l_{p}$ or $L_{p},$
with $p$ an even integer, any equivalent norm can be uniformly
approximated on bounded sets by analytic norms. Second, in
$X=c_{0}$ or $X=C\left( K\right),$ with $K$ a countable compact,
any equivalent norm can be uniformly approximated by analytic
norms on $X\backslash\left\{ 0\right\} .$

These results leave at least three important questions open.
Question one, is Kurzweil's result improvable? Is it necessary for
a Banach space $X$ to have a separating polynomial in order to
enjoy the property that every continuous function defined on $X$
can be uniformly approximated by real analytic functions? Fry and,
independently, Cepedello and Hajek have proved that every
uniformly continuous function defined on $c_0$ (which fails to
have a separating polynomial) \cite{F4} and, more generally, a
Banach space with a real analytic separating function \cite{CH},
can be uniformly approximated by real analytic functions. However,
the approximating functions they construct (again, refining
Kurzweil's technique, by employing a real analytic approximation
of a partition of unity) are not uniformly continuous, and in any
case the problem of approximating continuous functions (not
uniformly continuous) defined on $c_0$ by real analytic functions
remains open.

Question two. Given a $C^k$ function $f$ defined on a Banach space
$X$ with a separating polynomial, and a continuous
$\varepsilon:X\to (0, +\infty)$, is it possible to find a real
analytic function $g$ on $X$ such that
$\|D^{j}f-D^{j}g\|\leq\varepsilon$ for $j=0,1,..., k$? Nothing is
known about this problem, even in the case $k=1$ and $X$ being a
Hilbert space. It is worth mentioning that if we replace real
analytic with $C^\infty$ and consider only $k=1$, then the
question has an affirmative answer, which was found by Moulis
\cite{M} in the cases $X=\ell_p, c_0$ and with range any Banach
space (although her proof can be adapted to every Banach space
with an unconditional basis and a $C^\infty$ smooth bump function,
see \cite{AFGJL}), and by Hajek and Johanis \cite{HJ} very
recently for certain range spaces in the far more general case of
a separable Banach space $X$ with a $C^\infty$ smooth bump
function. And, interestingly enough, in both cases the solution
came as a corollary of a theorem on approximation of Lipschitz
functions by more regular Lipschitz functions. This leads us to
the following natural question, which is the subject matter of the
present paper.

Question three. Is it possible to approximate a Lipschitz function
$f$ defined on a Banach space $X$ having a separating  polynomial
by Lipschitz, real analytic functions $g$? And is it possible to
have the Lipschitz constants of the functions $g$ be of the order
of the Lipschitz constant of $f$ (that is $\textrm{Lip}(g)\leq
C\textrm{Lip}(g)$, where $C\geq 1$ is independent of $f$)? In such
case, does this hold for any number $C$ greater than one (that is,
is it possible to get real analytic approximations that almost
preserve the Lipschitz constant of the given function)? It is
worth noting that every approximation method based on constructing
real analytic approximations of partitions of unity in infinite
dimensions (such as \cite{K, F4, CH, J}) cannot work to give a
solution to this problem, the reason being that in order to make
the supports of the functions forming the partition locally finite
one has to give up the idea of those functions having a common
Lipschitz constant. That is, even if we consider $C^p$ smooth
functions instead of real analytic functions in this question, the
standard approximation technique based on the use of $C^p$ smooth
partitions of unity does not provide a solution to the problem.

A partial answer to this problem can be obtained by combining the
Deville-Fonf-Hajek results on real analytic approximation of
convex functions with the following theorem of M. Cepedello-Boiso
\cite{C}: a Banach space is superreflexive if and only if every
Lipschitz function can be approximated by differences of convex
functions which are bounded on bounded sets, uniformly on bounded
sets. It follows that, for $X=\ell_{p}$ or $L_{p}$, with $p$ an
even integer, and for every Lipschiz function $f:X\to\R$ there
exists a sequence of real analytic functions $g_n:X\to\R$ which
are Lipschitz on bounded sets, and such that
$\lim_{n\to\infty}g_n=f$, uniformly on bounded sets. However, this
method has two important disadvantages: 1) the approximation
cannot be made to be uniform on $X$, and 2) even on a fixed
bounded set $B$, we lose control on the Lipschitz constants of the
approximations $g_n$: in fact, one has
$\textrm{Lip}({g_{n}}_{|_{B}})\to \infty$ as $n\to\infty$.

A successful new approach to the question of approximating
Lipschitz functions on infinite-dimensional Banach spaces was
discovered by Fry in \cite{F1}, with the introduction of what one
can call {\em sup- partitions of unity} (see the following section
for a definition), a tool which has been thoroughly exploited in
\cite{HJ}.

Very recently, Fry and Keener \cite{FK} have constructed a real
analytic approximation of a sup partition of unity and have
employed it to show that, for every Banach space $X$ having a
separating polynomial, for every bounded open subset $U$ of $X$,
for every bounded Lipschitz function $f:U\to\R$, and for every
$\varepsilon>0$ there exists a Lipschitz, real analytic function
$g:U\to\R$ such that $|f-g|\leq\varepsilon$.

A disadvantage of the sup partition approach to approximation is
that it only works for bounded functions. In this paper we will
simplify the construction in \cite{FK}, and we will combine this
with some new tools in order to eliminate these restrictions and
obtain the following.

\begin{theorem}\label{analytic Lip approximation}
Let $X$ be a separable Banach space which admits a separating
polynomial. Then there exists a number $C\geq1$ such that for
every Lipschitz function $f:X\rightarrow\mathbb{R}$, and for every
$\varepsilon>0$ there exists a Lipschitz, real analytic function
$g:X\rightarrow \mathbb{R}$ such that $|f(x)-g(x)|\leq\varepsilon$
for all $x\in X$, and $\text{Lip}(g)\leq C\text{Lip}(f)$.
\end{theorem}

As a matter of fact the proof of this result works for any Banach
space $X$ having a separating function with a Lipschitz
holomorphic extension to a uniformly wide neighborhood of $X$ in
the complexification $\widetilde{X}$.

\begin{definition}
A separating function $Q$ on a Banach space $X$ is a function
$Q:X\to [0, +\infty)$ such that $Q(0)=0$ and there exists $M, m>0$
such that $Q(x)\geq m\|x\|$ whenever $\|x\|\geq M$.
\end{definition}
\noindent It is clear that $X$ has a real analytic Lipschitz
separating function (with a holomorphic extension to a uniformly
wide neighborhood of $X$ in $\widetilde{X}$) if and only if $X$
has a real analytic Lipschitz function $Q$ which satisfies the
above definition with $M=m=1$ (and with a holomorphic extension to
a uniformly wide neighborhood of $X$ in $\widetilde{X}$). So we
also have the following.

\begin{theorem}\label{characterization of analytic Lip approximation}
Let $X$ be a separable Banach space having a separating function
$Q:X\to [0, \infty)$ with a Lipschitz holomorphic extension
$\widetilde{Q}$ defined on a set of the form
$\{x+iy\in\widetilde{X} : x, y\in X, \|y\|<\delta\}$ for some
$\delta>0$. Then there exists a number $C\geq1$ such that for
every Lipschitz function $f:X\rightarrow\mathbb{R}$, and for every
$\varepsilon>0$ there exists a Lipschitz, real analytic function
$g:X\rightarrow \mathbb{R}$ such that $|f(x)-g(x)|<\varepsilon$
for all $x\in X$, and $\text{Lip}(g)\leq C\text{Lip}(f)$.
\end{theorem}

We will prove (see Lemma \ref{separating function Q lemma} below)
that every Banach space with a separating polynomial (that is a
polynomial $P:X\to\R$ such that $P(0)=0<\inf\{P(x) : \|x\|=1\}$)
also has a Lipschitz, real analytic separating function with a
holomorphic extension to a uniformly wide neighborhood of $X$ in
$\widetilde{X}$. We do not know if the converse is true. A natural
related question is: does the space $c_0$ have such a Lipschitz
separating function?

By using Theorem \ref{analytic Lip approximation} and part of its
proof (namely Lemma \ref{construction of the sup partition of
unity} below) one can also prove the following.
\begin{theorem}\label{derivative approximation}
Let $X$ be a separable Banach space which admits a separating
polynomial. Let $f:X\rightarrow\mathbb{R}$ be bounded, Lipschitz,
and $C^{1}$ with uniformly continuous derivative. Then for each
$\varepsilon>0,$ there exists a real analytic function
$g:X\rightarrow\mathbb{R}$ with $\left\vert g-f\right\vert
<\varepsilon$ and $\left\Vert g^{\prime}-f^{\prime}\right\Vert
<\varepsilon.$
\end{theorem}
\noindent We give a detailed proof of this result in \cite{AFK3}.

\medskip

Finally, by combining the Lasry-Lions sup-inf convolution
regularization technique \cite{LaLi} with the preceding Theorem
and with some of the techniques developed for the proof of Theorem
\ref{analytic Lip approximation}, we will obtain the following
improvement of our main result in the most important case; namely,
if $X$ is a Hilbert space, then the real analytic approximations
$g$ can be made to almost preserve the Lipschitz constant of the
given function $f$.

\begin{theorem}\label{Lip approximation on Hilbert space}
Let $X$ be a separable Hilbert space, and $f:X\to\mathbb{R}$ a
Lipschitz function. Then for every $\varepsilon>0$ there exists a
Lipschitz, real analytic function $g:X\rightarrow \mathbb{R}$ such
that $|f(x)-g(x)|\leq \varepsilon$ for all $x\in X$, and
$\text{Lip}(g)\leq \text{Lip}(f)+\varepsilon$.
\end{theorem}

\bigskip

\section{A brief outline of the proof}

The proof of our main result is rather long, and very technical at
some points so, for the reader's convenience, we will next explain
the main ideas of our construction  (which we here intentionally
oversimplify in order not to be burdened by important, but not
very meaningful precision and notation).

As said in the introduction, we will show in Lemma \ref{separating
function Q lemma} that every Banach space $X$ with a separating
polynomial $p$ of degree $n$ has a Lipschitz, real analytic
separating function. This is done as follows: such a Banach space
always has a $2n$-degree homogeneous polynomial $q$ such that
$\Vert x\Vert ^{2n}\leq q(x)\leq K\Vert x\Vert^{2n}$ for all $x\in
X$. Then the function $Q:X\rightarrow \lbrack0,+\infty)$ defined
by
\[
Q(x)=\left(  1+q(x)\right)  ^{\frac{1}{2n}}-1
\]
is real analytic, Lipschitz, and separating.

The next step will be taking a dense sequence $\{x_n\}$ in $X$ and
constructing an equi-Lipschitz, real analytic analogue of a sup
partition of unity $\{\varphi_n\}_{n\in\N}$ which is subordinated
to the covering $X=\bigcup_{n=1}^{\infty}D_{Q}(x_{n}, 4)$, where
$D_{Q}(x_n, 4)=\{x\in X : Q(x-x_n)<4\}$. By this we mean a
collection of real analytic functions $\varphi_n :X\to [0,2]$,
with holomorphic extensions $\widetilde{\varphi}_n$ defined on an
open neighborhood $\widetilde{V}$ of $X$ in the complexification
$\widetilde{X}$ of $X$, such that:
\begin{enumerate}
\item The collection $\{\varphi_{n} :X\to[0,2] \, | \, n\in\mathbb{N}\}$ is
equi-Lipschitz on $X$, with Lipschitz constant $L=2\text{Lip}(Q)$.
\item For each $x\in X$ there exists $m=m_{x}\in\mathbb{N}$
with $\varphi_{m}(x)>1$.
\item For every $x\in X$ the set $\{n\in\N :
\varphi_n(x)>\varepsilon\}$ is finite.
\item $0\leq\varphi_{n}(x)\leq\varepsilon$ for all $x\in X\setminus
D_{Q}(x_{n}, 4)$.
\item The function $\widetilde{V}\ni z\mapsto
\{\alpha_n\widetilde{
\varphi}_n(z)\}_{n=1}^{\infty}\in\widetilde{c_0}$ is holomorphic
for every sequence $\{\alpha_n\}$ such that $1\leq \alpha_n \leq
1001$ for all $n$.
\end{enumerate}

Next, there is a real analytic norm $\lambda:c_0\to [0, \infty)$
which satisfies $\|y\|_{\infty}\leq \lambda(y)\leq
2\|y\|_{\infty}$ for every $y\in c_0$. Assuming $f:X\to [1, 1001]$
is $1$-Lipschitz, define a function $g:X\to\mathbb{R}$ by
    $$
g(x)=\frac{\lambda(\{f(x_n)\varphi_{n}(x)\}_{n=1}^{\infty})}{\lambda(\{\varphi_{n}(x)\}_{n=1}^{\infty})}
    $$
(the sup-partition of unity approach to approximation consists in
replacing the usual locally finite sum in a classical partition of
unity by taking the norm $\lambda$. Since $\lambda$ is equivalent
to the supremum norm of $c_0$, the function $g$ roughly behaves
like $\sup_{n}\{f(x_n)\varphi_n(x)\}/\sup_{n}\{\varphi_n(x)\}$,
which is a Lipschitz function approximating $f$). It can be
checked that $g$ is real analytic,
$8518492\textrm{Lip}(Q)$-Lipschitz, and $|f-g|\leq 8$. Moreover,
given $\delta>0$, there exists a neighborhood
$\widetilde{U}_{\delta}$ of $X$ in $\widetilde{X}$, {\em
independent of $f$ but dependent on the interval $[1, 1001]$ and
the functions $\varphi_n$}, such that $g$ has a holomorphic
extension $\widetilde{g}:\widetilde{U}_{\delta}\to\C$ satisfying
    $$
|\widetilde{g}(z)-g(x)|\leq\delta \textrm{ for all }
z=x+iy\in\widetilde{U}_{\delta}.
    $$
Proving such independence of $\widetilde{U}_{\delta}$ from $f$ is
a delicate matter that will be tackled with the help of a
refinement of the main result of \cite{CHP}: an estimation of the
domain of existence of the holomorphic solutions to a family of
complex implicit equations depending on a parameter (see
Proposition \ref{refinement of CHP} below).

If we were only interested in approximating bounded functions by
Lipschitz, real analytic functions and we did not care about the
Lipschitz constant of the approximations, a replacement of the
interval $[1, 1001]$ with a suitable translation of the range of
$f$ in the above argument would finish our proof (up to scaling).
However, as the size of the interval increases, so does the
Lipschitz constant of $f$, and inversely, the size of
$\widetilde{U}_{\delta}$ decreases. In order to prove our main
result in full generality we have to work harder.

Up to scaling, the above argument shows the following: there
exists $C\geq 1$ (depending only on $X$) such that, for every
$\delta>0$ there is an open neighborhood $\widetilde{U}_{\delta}$
of $X$ in $\widetilde{X}$ such that, for every Lipschitz function
$f:X\to[0,1]$ with $\textrm{Lip}(f)\leq 1$, there exists a real
analytic function $g:X\to\mathbb{R}$, with holomorphic extension
$\widetilde {g}:\widetilde{U}_{\delta}\to\mathbb{C}$, such that
\begin{enumerate}
\item $|f(x)-g(x)|\leq 1/10$ for all $x\in X$.
\item $g$ is Lipschitz, with $\textrm{Lip}(g)\leq
C\textrm{Lip}(f)$.
\item $|\widetilde{g}(x+iy)-g(x)|\leq \delta$ for all $z=x+iy\in\widetilde{U}_{\delta}$.
\end{enumerate}

Now, given a $1$-Lipschitz, bounded function $f:X\to [0,
+\infty)$, we can cut its graph into pieces, namely we may
 define, for $n\in\N$, the functions
\[
f_{n}(x)=
\begin{cases}
f(x)-n+1 & \text{ if }n-1\leq f(x)\leq n,\\
0 & \text{ if }f(x)\leq n-1,\\
1 & \text{ if }n\leq f(x).
\end{cases}
\]
The functions $f_{n}$ are clearly $1$-Lipschitz and take values in
the interval $[0,1]$, so, for $\delta>0$ (to be fixed later on)
there exist a neighborhood $\widetilde{U}_{\delta}$ of $X$ in
$\widetilde{X}$ and $C$-Lipschitz, real analytic functions
$g_{n}:X\rightarrow\mathbb{R}$, with holomorphic extensions
$\widetilde{g}_n :\widetilde{U}_{\delta}\to\C$, such that for all
$n\in\mathbb{N}$ we have that $g_{n}$ is $C$-Lipschitz,
$|g_{n}-f_{n}|\leq 1/8$, and $|\widetilde{g}_n(x+iy)-g_n(x)|\leq
\delta$ for all $z=x+iy\in\widetilde{U}_{\delta}$.

Observe that the function $X\ni
x\mapsto\{f_{n}(x)\}_{n=1}^{\infty}$ takes values in
$\ell_\infty$, and more precisely in the image of the path
$\beta:[ 0,+\infty) \rightarrow \ell_{\infty}$ defined by
\[
\beta\left(  t\right)  =\left(  1,\cdots,1,\underset{n^{\text{th}%
}\ \text{place}}{\underbrace{t-n+1}},0,0,\cdots
\right)=\sum_{j=1}^{n-1}e_j +(t-n+1)e_n  \text{ if }  n-1\leq
t\leq n.
\]
The path $\beta$ is a $1$-Lipschitz injection of $[0, +\infty)$
into $\ell_{\infty}$, with a uniformly continuous (but not
Lipschitz) inverse $\beta^{-1}:\beta([0,+\infty))\to [0,
+\infty)$.

Define a uniformly continuous (not Lipschitz) function $h$ on the
path $\beta$ by $h\left( \beta\left( t\right) \right) =t$ for all
$t\geq 0$, that is $h(y)=\beta^{-1}(y)$ for $y\in\beta([0,
+\infty))$. Then we have $f\left( x\right) =h\left(
\{f_{n}(x)\}_{n=1}^{\infty}\right)$.

We will construct an open tube of radius $1/8$ (with respect to
the supremum norm $\|\cdot\|_{\infty}$) around the path $\beta$ in
$\ell_{\infty}$, and a real-analytic approximate extension (with
bounded derivative) $H$ of the function $h$ defined on this tube.
Then, since $|g_n-f_n|\leq 1/10$ and $\{f_{n}(x)\}_{n=1}^{\infty}$
takes values in the path $\beta$, the function
$\{g_{n}(x)\}_{n=1}^{\infty}$ will take values in this tube, and
therefore $g(x):=H(\{g_{n}(x)\}_{n=1}^{\infty})$ will approximate
$H(\{f_{n}(x)\}_{n=1}^{\infty})$, which in turn approximates
$h(\{f_{n}(x)\}_{n=1}^{\infty})=f(x)$. Besides, since $H$ has a
bounded derivative on the tube and the functions $g_n$ are
$C$-Lipschitz then $g$ will be $CM$-Lipschitz, where $M$ is an
upper bound of $DH$ on the tube. Moreover, we will show that there
exist $\delta>0$ (this is the $\delta$ we had to fix) and a
holomorphic extension $\widetilde{H}$ of $H$ so that if
$v\in{\ell}_{\infty}$ satisfies $\|v\|_{\infty}\leq \delta$ then
$$
|H(u+iv)-H(u)|\leq 1.
$$
We will call $H$ a real analytic {\em gluing tube function}.

Thus, resetting $C$ to $CM$, we will deduce that, there exists an
open neighborhood $\widetilde{U}:=\widetilde{U}_{\delta}$ of $X$
in $\widetilde{X}$ such that, for every $1$-Lipschitz, bounded
function $f:X\to\R$, there exists a real analytic function
$g:X\to\R$, with holomorphic extension
$\widetilde{g}:\widetilde{U}\to\C$, such that
\begin{enumerate}
\item $|f(x)-g(x)|\leq 1$ for all $x\in X$.
\item $g$ is $C$-Lipschitz.
\item $|\widetilde{g}(x+iy)-g(x)|\leq 1$ for all $z=x+iy\in\widetilde{U}$.
\end{enumerate}

The last step of the proof consists in passing from bounded to
unbounded functions. This can be achieved by constructing a real
analytic approximation $\theta_n(Q(x))$ to a partition of unity
subordinated to a covering of $X$ consisting of crowns $C_n=\{x\in
X : 2^{n-1}<Q(x)<2^{n+1}\}$ of rapidly increasing diameter (so
that $\sum_{n=1}^{\infty}\textrm{Lip}(\theta_n)<3$), and by
approximating each restriction of $f$ to $C_n$ by a real analytic
$C$-Lipschitz function $g_n$, in order to define
    $$
g(x)=\sum_{n=1}^{\infty}\theta_{n}(Q(x))g_n(x),
    $$
which we will check is a $5C$-Lipschitz real analytic
approximation of $f$. Up to scaling, our main theorem will then be
proved.

\bigskip

\section{Notation and basic definitions. Complexifications.}

Our notation is standard, with $X$ denoting a Banach space, and an
open ball with centre $x$ and radius $r$ denoted $B_{X}\left( x,
r\right)$ or $B(x,r)$ when the space is understood.  If $\left\{
f_{j}\right\} _{j}$ is a sequence of Lipschitz functions defined
on $X,$ then we will at times say this family is
\textit{equi-Lipschitz} if there is a common Lipschitz constant
for
all $j.$ A \textit{homogeneous polynomial of degree }$n$ is a map,\textit{\ }%
$P:X\rightarrow\mathbb{R},$ of the form $P\left(  x\right)
=A\left( x,x,...,x\right)  ,$ where $A:X^{n}\rightarrow\mathbb{R}$
is $n-$multilinear and continuous. For $n=0$ we take $P$ to be
constant. A \textit{polynomial of degree }$n$ is a sum
$\sum_{i=0}^{n}P_{i}\left(  x\right)  ,$ where the $P_{i}$ are
$i$-homogeneous polynomials.

Let $X$ be a Banach space, and $G\subset X$ an open subset. A
function $f:G\rightarrow\mathbb{R}$ is called \textit{analytic }if
for every $x\in G,$
there are a neighbourhood $N_{x},$ and homogeneous polynomials $P_{n}%
^{x}:X\rightarrow\mathbb{R}$ of degree $n$, such that
\[
f\left(  x+h\right)  =\sum_{n\geq0}P_{n}^{x}\left(  h\right)
\;\text{provided\ }x+h\in N_{x}.
\]
Further information on polynomials may be found, for example, in
$\left[ \text{SS}\right]  .$

For a Banach space $X,$ we define its (Taylor) complexification
$\widetilde{X}=X\bigoplus iX$ with norm
\[
\left\|  x+iy\right\|  _{\widetilde{X}%
}=\sup_{0\leq\theta\leq2\pi}\left\|  \cos\theta\ x-\sin\theta\
y\right\| _{X}=\sup_{T\in X^{*}, \|T\|\leq 1}\sqrt{T(x)^2
+T(y)^2}.
\]
We will sometimes denote, for $z=x+iy\in\widetilde{X}$,
$x:=\text{Re}\,z, \,\, y:=\text{Im}\,z$. If $L:E\to F$ is a
continuous linear mapping between two real Banach spaces then
there is a unique continuous linear extension
$\widetilde{L}:\widetilde{E}\to\widetilde{F}$ of $L$ (defined by
$\widetilde{L}(x+iy)=L(x)+iL(y)$) such that
$\|\widetilde{L}\|=\|L\|$. For a continuous $k$-homogeneous
polimomial $P:E\to\R$ there is also a unique continuous
$k$-homogeneous polinomial $\widetilde{P}:\widetilde{E}\to\C$ such
that $\widetilde{P}=P$ on $E\subset\widetilde{E}$, but the norm of
$P$ is not generally preserved: one has that
$\|\widetilde{P}\|\leq 2^{k-1}\|P\|$. It follows that if $q\left(
x\right) $ is a continuous polynomial on $X,$ there is a unique
continuous polynomial $\widetilde {q}\left( z\right)
=\widetilde{q}\left( x+iy\right)$ on $\widetilde{X}$ where for
$y=0$ we have $\widetilde{q}=q.$ It also follows that if the
Taylor series $f(x+h)=\sum_{n=0}^{\infty}D^{n}f(x)(h)/n!$ of a
real analytic function $f:X\to\R$ at a point $x$ has radius of
convergence $r_x$ then the series
$\sum_{n=0}^{\infty}\widetilde{D^{n}f(x)}(h)/n!$ has radius of
convergence $r_x/2e$ (see for instance \cite[Lemma 0]{CH}).
Consequently $f$ has a holomorphic extension $\widetilde{f}$ to
the neighborhood $\{x+iy : x, y\in X,
\|x+iy\|_{\widetilde{X}}<r_x\}$ of $X$ in $\widetilde{X}$.

We will also use the fact that for this complexification procedure
the complexifications $\widetilde{c_0}, \widetilde{\ell}_{\infty}$
of the real Banach spaces $c_0$ and $\ell_{\infty}$ are just the
usual complex versions of these spaces. That is, we have
$\widetilde{c}_{0}=\{\left\{  z_{j}\right\}  : z_{j}
\in\mathbb{C}$, $ \lim_{j\to\infty}z_{j} = 0\}$, with norm
$\left\Vert z\right\Vert _{\widetilde{c}_{0}}=\left\Vert \left\{
z_{j}\right\} \right\Vert _{\widetilde{c}_{0}}=\max_{j}\left\{
\left\vert z_{j}\right\vert \right\}  ,$ and similarly for
$\widetilde {l}_{\infty}.$ For more information on
complexifications (and polynomials) we recommend \cite{MST}.

In the sequel, all extensions of functions from $X$ to
$\widetilde{X},$ as well as subsets of $\widetilde{X},$ will be
embellished with a tilde. All smoothness properties of a norm or
Minkowski functional are assumed implicitly to not include the
point $0$.

\bigskip

\section{Preliminary results: the Preiss norm and separating functions}

\medskip

\subsection{The Preiss norm}

As developed in \cite{FPWZ}, there is a real analytic norm on
$c_{0}$ (hereafter referred to as the Preiss norm, $\|\cdot\|$)
that is equivalent to the canonical supremum norm
$\|\cdot\|_{\infty}$. Let us recall the construction.  Let
$C:c_{0}\rightarrow\mathbb{R}$ be given by
$C(\{x_{n}\})=\sum_{n=1}^{\infty}\left(  x_{n}\right) ^{2n}$. Let
$W=\{x\in c_{0}:C(x)\leq1\}.$ Then $\left\Vert \cdot\right\Vert $
is the Minkowski functional of $W$; that is, $\left\Vert
x\right\Vert $ is the solution for $\lambda$ to $C\left(
\lambda^{-1}x\right)  =1.$ The Preiss norm is analytic at all
non-zero points in $c_{0}.$ To see this, let us define the
function $\widetilde{C}:V\rightarrow\mathbb{C}$ by
$\widetilde{C}\left( \left\{  z_{n}\right\}  \right)
=\sum_{n=1}^{\infty}\left(  z_{n}\right) ^{2n}$ where $V$ is the
subset of $\widetilde{l}_{\infty}$ for which the series converges.
Then $\widetilde{C}$ is analytic at each $z\in\widetilde {c}_{0}$.
Indeed, the partial sums are analytic as a consequence of the
analyticity of the (complex linear) projection functions
$p_{j}(\{z_{i}\})=z_{j}$. Since the series in the definition of
$\widetilde{C}$ converges locally uniformly at each
$z\in\widetilde{c}_{0}$ the analyticity of $\widetilde{C}$ on
$\widetilde {c}_{0}$ follows. Also, for $z\in\widetilde{c}_{0}$
sufficiently close to $c_{0}$ and $\lambda\in \C\setminus\{0\}$
sufficiently close to $\R\setminus\{0\}$ we have
$\frac{\partial\widetilde{C}\left( \lambda^{-1}z\right)
}{\partial\lambda }\neq0,$ hence one can apply the complex
Implicit Function Theorem (see e.g. \cite{Cartan}, or \cite{D}
page 265, where the real result for Banach spaces is easily
extended to the analytic case) to $F\left(  z,\lambda\right)
=\widetilde{C}\left( \lambda^{-1}z\right)  -1$ to obtain a unique
holomorphic solution $\widetilde{\lambda}\left(
z\right)  $ to $F\left(  z,\lambda\right)  =0,$ with $\lambda:=\widetilde{\lambda}\mid_{c_{0}%
}=\left\Vert \cdot\right\Vert$, defined on a neighborhood of $c_0
\setminus\{0\}$ in $\widetilde{c_0}$. Now if $x=\{x_{n}\}$
satisfies $\left\Vert x\right\Vert _{\infty}=1,$ then
$\sum_{n=1}^{\infty}\left( \left\Vert x\right\Vert
^{-1}x_{n}\right)  ^{2n}=1$ implies $\left\Vert x\right\Vert
\geq1.$ On the other hand, if $\left\Vert x\right\Vert
_{\infty}=1/2,$ then $C\left(  x\right)
\leq\sum_{n=1}^{\infty}(1/2)^{2j}<1,$ implying $\left\Vert
x\right\Vert <1$. Hence, $\left(  1/2\right) \left\Vert
x\right\Vert \leq\left\Vert x\right\Vert _{\infty}\leq\left\Vert
x\right\Vert $ for all $x$ in $c_{0}$. We shall use the above
notation throughout this article.

\medskip

\subsection{Separating polynomials and separating Lipschitz analytic functions}

Let $X$ be a Banach space. A \textit{separating polynomial} on $X$
is a polynomial $q$ on $X$ such that $0=q(0)<\inf\{|q(x)|:x\in
S_{X}\}$. It is known [FPWZ] that if $X$ is superreflexive and
admits a $C^{\infty}$-smooth bump function then $X$ admits a
separating polynomial. The following lemma makes precise,
observations of Kurzweil in [K].


\begin{lemma}
Let $X$ be a Banach space with a separating polynomial of degree
$n$. Then there exist $K>1$ and a $2n$ degree homogeneous
separating polynomial $q:X\rightarrow \lbrack0,+\infty)$ such that
$\Vert x\Vert ^{2n}\leq q(x)\leq K\Vert x\Vert^{2n}$ for all $x\in
X$.
\end{lemma}


\begin{proof}
We may suppose that $p=\sum_{i=1}^{n}p_{i},$ where $p_{i}$ is
$i$-homogeneous for $1\leq i\leq n$. Define
$q=p_{1}^{2n}+p_{2}^{2n/2}+\cdots p_{n}^{2}$. Note that $q$ is
$2n$-homogeneous. As $p$ is separating, there is some $\eta>0$
such that $q(x)\geq\eta$ for all $x\in S_{X}$. By scaling, we may
assume that $\eta=1$. Then from the $2n$-homogeneity, $q\left(
x\right)  \geq\left\Vert x\right\Vert ^{2n}.$ Now, as $q$ is
continuous at $0,$ there is $\delta>0$ with $\left\Vert
x\right\Vert \leq\delta$ implies $q\left(  x\right)  \leq1,$ so
again from $2n$-homogeneity, $1\geq q\left(  \frac{\delta
x}{\left\Vert x\right\Vert }\right)  =q\left(  x\right)
\frac{\delta^{2n}}{\left\Vert x\right\Vert ^{2n}},$ and we are
done.
\end{proof}

\medskip

Note that a separating polynomial, even though is Lipschitz on
every bounded set, is never Lipschitz on all of $X$. In the proof
of our main results we will require the existence of a globally
Lipschitz, real analytic separating function on $X$. The following
lemma provides us with such separating functions for every Banach
space having a separating polynomial.

\begin{lemma}
\label{separating function Q lemma} Let $X$ be a Banach space with
a separating polynomial $q$ as in the previous Lemma. Then the
function $Q:X\rightarrow \lbrack0,+\infty)$ defined by
\[
Q(x)=\left(  1+q(x)\right)  ^{\frac{1}{2n}}-1
\]
is real analytic, satisfies $\inf_{\Vert x\Vert\geq
1}Q(x)>0=Q(0)$, and is Lipschitz on all of $X$. Moreover, there is
some $\delta_{Q}>0$ such that $Q$ has an holomorphic extension
$\widetilde{Q}$ to the open strip $W_{\delta_{Q}}:=\{x+z:x\in
X,z\in \widetilde{X},\Vert z\Vert_{\widetilde{X}}<\delta_Q\}$ in
$\widetilde{X}$ such that $\widetilde{Q}$ is still Lipschitz on
$W_{\delta_{Q}}$. Finally, we have that
    $$
Q(x)<4r \implies \|x\| <8r , \, \textrm{ for all } \, r\geq 1,
x\in X.
    $$
In particular $Q(x)\geq \frac{1}{2}\|x\|$ for $\|x\|\geq 8$ and
$Q$ is separating.
\end{lemma}
\begin{proof}
Since the function $w\mapsto w^{1/2n}$ is well defined (taking the
usual branch of log) and holomorphic on the half-plane
$\text{Re}\,w\geq1/2$ in $\mathbb{C}$, it is clear that
\[
\widetilde{Q}(x+z)=\left(  1+\widetilde{q}(x+z)\right)  ^{1/2m}%
\]
will be well defined and holomorphic on such a strip
$W_{\delta_Q}$ as soon as we find $\delta=\delta_Q>0$ small enough
so that $\operatorname{Re}\left( 1+\widetilde{q}(x+z)\right)
\geq1/2$ whenever $x\in X,z\in\widetilde{X},\Vert
z\Vert\leq\delta$.

Let $A$ (resp. $\widetilde{A}$) denote a $2n$-multilinear form on
$X$ (resp. $\widetilde{X}$) such that $q(x)=A(x, ..., x)$ (resp.
$\widetilde {q}(w)=\widetilde{A}(w, ..., w)$). We have, for $x\in
X$ and $z\in \widetilde{X}$ with $\|z\|\leq\delta<1$
($\delta=\delta_{Q}$ is to be fixed later),
\begin{align*}
&  \text{Re}\,\widetilde{q}(x+z)=\text{Re}\,A(x+z, x+z, ..., x+z)\\
&  =A(x, x, ...,x)+\text{ terms of the form } \text{Re}\,A(y_{1},
...,
y_{2n}),\\
&  \text{where } y_{j}\in\{x,z\} \text{ for all } j, \text{ and }
y_{k}=z
\text{ for at least one } k\\
&  \geq A(x, ..., x)-\sum_{j=1}^{2n}\|\widetilde{A}\|\binom{2n}{j}%
\|x\|^{2n-j}\|z\|^{j}\\
&  \geq\|x\|^{2n}-\|A\|\sum_{j=1}^{2n}\binom{2n}{j}\|x\|^{2n-j}\delta^{j}\\
&   \geq\|x\|^{2n}-\delta\|A\|\sum_{j=1}^{2n}\binom{2n}{j}\|x\|^{2n-j}\\
&  \geq\|x\|^{2n}-\delta\|A\|\left( \|x\|+1\right) ^{2n}.
\end{align*}
Hence
\[
\text{Re}\,(1+\widetilde{q}(x+z))\geq\frac{1}{2}+\frac{1}{2}+\|x\|^{2n}%
-\delta\|A\|\left( \|x\|+1\right) ^{2n}.
\]
Since the real function
\[
f(t)=\frac{\frac{1}{2}+t^{2n}}{(|t|+1)^{2n}}
\]
is continuous, positive and satisfies $\lim_{|t|\to\infty}f(t)=1$,
there
exists $\alpha\in (0, 1)$ such that $f(t)\geq\alpha$, and therefore $1+t^{2n}%
-\alpha(|t|+1)^{2n}\geq1/2$ for all $t\in\mathbb{R}$. If we set
$$\delta
=\frac{1}{2}\min\{1, \alpha/\|A\|\},$$ this implies that $1+\|x\|^{2n}%
-\delta\|A\|\left( \|x\|+1\right) ^{2n}\geq1/2$ for all $x\in X$,
and therefore
\[
\text{Re}\,(1+\widetilde{q}(x+z))\geq\frac{1}{2}, \textrm{ for all
} x+z\in W_{\delta},
\]
and the function $\widetilde{Q}$ is holomorphic on $W_{\delta}$.

Now let us check that $\widetilde{Q}$ is Lipschitz on
$W_{\delta}$. By the same estimation as above, for $x\in X$,
$z\in\widetilde{X}$ with $\Vert z\Vert\leq\delta$, we have
\[
|1+\widetilde{q}(x+z)|\geq\text{Re}\,(1+\widetilde{q}(x+z))\geq
1+\Vert x\Vert^{2n}-\delta\Vert A\Vert\left(  \Vert
x\Vert+1\right) ^{2n}\geq1/2,
\]
and on the other hand, the derivative of $\widetilde{q}$ being a
$(2n-1)$-homogeneous polynomial,
\[
\left\vert \widetilde{q}^{\prime}(x+z)\right\vert \leq\left\Vert
\widetilde {q}^{\prime}\right\Vert \left(  \left\Vert x\right\Vert
+\left\Vert z\right\Vert \right)  ^{2n-1}\leq\left\Vert
\widetilde{q}^{\prime}\right\Vert \left(  \left\Vert x\right\Vert
+\delta\right)  ^{2n-1}.
\]
By combining these two inequalities we get, for $x+z\in
W_{\delta}$,
$$
\Vert\widetilde{Q}\,^{\prime}(x+z)\Vert=\frac{\Vert\widetilde{q}^{\prime
}(x+z)\Vert}{\left\vert 2n\left(  1+\widetilde{q}(x+z)\right)
^{1-\frac {1}{2n}}\right\vert }\leq\frac{\left\Vert
\widetilde{q}^{\prime}\right\Vert \left(  \left\Vert x\right\Vert
+\delta\right)  ^{2n-1}}{\left( 1+ \Vert x\Vert^{2n}-\delta\Vert
A\Vert\left(  \Vert x\Vert+1\right)  ^{2n}\right)
^{\frac{2n-1}{2n}}}.
$$
Since the function
\[
\mathbb{R}\ni t\mapsto\frac{\left\Vert
\widetilde{q}^{\prime}\right\Vert \left(  \left\vert t\right\vert
+\delta\right)  ^{2n-1}}{\left( 1+|t|^{2n}-\delta\Vert
A\Vert\left( |t|+1\right)  ^{2n}\right)  ^{\frac
{2n-1}{2n}}}\in\mathbb{R}%
\]
is bounded on $\mathbb{R}$, it follows that the derivative
$\widetilde{Q}^{\prime}$ is bounded on the convex set
$W_{\delta}$, and therefore the function $\widetilde{Q}$ is
Lipschitz on $W_{\delta}$.

Finally, assume that $\|x\|\geq 8r$, $r\geq 1$. Since $q(x)\geq
\|x\|^{2n}$ and the function $h(t)=(1+t^{2n})^{1/2n}-1$ is
increasing, we have
    $$
Q(x)\geq h(\|x\|)\geq h(8r).
    $$
Now, for $r\geq 1$ we have
    $$
(1+4r)^{2n}\leq 2^{2n} (4r)^{2n}\leq 1+(8r)^{2n},
    $$
from which it follows that
    $$
h(8r)\geq 4r,
    $$
and therefore $Q(x)\geq 4r$.
\end{proof}


\bigskip

\section{The proof of the main result}

\medskip

In the sequel we will be using the following notation: for the
real analytic and Lipschitz separating function $Q$ constructed in
Lemma \ref{separating function Q lemma}, and for every $x\in X$,
$r>0$, we define the $Q$-bodies
\[
D_{Q}(x,r)=\{y\in X:Q(y-x)<r\}.
\]

\medskip

\subsection{Real analytic sup-partitions of unity.}

The following Lemma provides us, for any given $r>0,
\varepsilon\in (0,1)$, with a real-analytic analogue of {\em
sup-partitions of unity} (see \cite{F1, HJ}) which, one could say,
are {\em $\varepsilon$-subordinated} to a covering of the form
$X=\bigcup_{j=1}^{\infty}D_{Q}(x_{j}, 4r)$.


\begin{lemma}
\label{construction of the sup partition of unity} Let
$\widetilde{V}=W_{\delta_Q}$ be an open strip around $X$ in
$\widetilde{X}$ in which the function $\widetilde{Q}$ of Lemma
\ref{separating function Q lemma} is defined. Given $r, K, \eta>0,
\varepsilon\in(0,1),$ and a covering of the form
$X=\bigcup_{j=1}^{\infty}D_{Q}(x_{j}, r)$, there exists a sequence
of holomorphic functions
$\widetilde{\varphi}_{n}=\widetilde{\varphi}_{n,r, \varepsilon}
:\widetilde{V}\to\mathbb{C}$, whose restrictions to $X$ we denote
by $\varphi_{n}=\varphi_{n, r, \varepsilon}$, with the following
properties:

\begin{enumerate}
\item The collection $\{\varphi_{n, r, \varepsilon} :X\to[0,2] \, | \, n\in\mathbb{N}\}$ is
equi-Lipschitz on $X$, with Lipschitz constant
$L=2\text{Lip}(Q)/r$ (independent of $\varepsilon)$.

\item $0\leq\varphi_{n, r, \varepsilon}(x)\leq 1+\varepsilon$ for all $x\in X$.

\item For each $x\in X$ there exists $m=m_{x,r}\in\mathbb{N}$ (independent of $\varepsilon)$
with $\varphi_{m, r, \varepsilon}(x)>1$. Besides, $m_{x,s}\geq
m_{x,r}$ whenever $r\geq s$.

\item $0\leq\varphi_{n, r, \varepsilon}(x)\leq\varepsilon$ for all $x\in X\setminus
D_{Q}(x_{n}, 4r)$.

\item[{(4')}] $\|\varphi_{n, r, \varepsilon}(x)\|\leq\varepsilon$
for all $x\in X\setminus D_{Q}(x_{n}, 5r)$.

\item For each $x\in X$ there exist $\delta_{x, r} >0$, $a_{x, r}>0$, and $n_{x, r}\in\mathbb{N}$
(independent of $\varepsilon$) such that
$$|\widetilde {\varphi}_{n, r, \varepsilon} (x+z)|<\frac{1}{n! a_x^{n}M_n}  \,\,
\textrm{ for } \,  n>n_{x, r}, \, z\in \widetilde{X} \textrm{ with
} \|z\|_{\widetilde{X}}<\delta_{x, r},
$$
where $M_n:=e^{2K^2\eta}(1+\|x_n\|)$.

\item For each $x\in X$ there exists $\delta_{x, r} >0$ (independent of $\varepsilon$)
and $n_{x, r, \varepsilon}\in\mathbb{N}$ such that for
$\|z\|_{\widetilde{X}}<\delta_{x, r}$ and $n>n_{x, r,\varepsilon}$
we have $|\widetilde {\varphi}_{n, r, \varepsilon}
(x+z)|<\varepsilon$.

\item For each $x\in X$ there exists $\delta_{x,r,\varepsilon}$
such that
$$|\widetilde {\varphi}_{n, r, \varepsilon}
(x+z)|\leq 1+2\varepsilon \,\, \textrm{ for } n\in\N, r\geq 1, \,
z\in \widetilde{X} \textrm{ with }\|z\|_{\widetilde{X}}\leq
\delta_{x, r, \varepsilon}.$$
\end{enumerate}
Moreover, in the case when $r\geq 1$, the numbers $\delta_{x, r},
a_{x, r}, n_{x, r}, \delta_{x, r, \varepsilon}$ and $n_{x, r,
\varepsilon}$ can be assumed to be independent of $r$.
\end{lemma}
\begin{proof}
Define subsets $A_{1,r}=\{y_{1}\in\mathbb{R}: -1\leq y_{1}\leq
4r\}$, and, for $n\geq2$,
\[
A_{n,r} =\{y=\{y_{j}\}_{j=1}^{n}\in\ell^{n}_{\infty} : -1-r\leq
y_{n}\leq4r, \, 2r\leq y_{j}\leq M_{n,r}+2r \text{ for } 1\leq
j\leq n-1\},
\]
\[
A^{\prime}_{n,r}=\{y=\{y_{j}\}_{j=1}^{n}\in\ell^{n}_{\infty} : -1\leq y_{n}%
\leq3r, \, 3r\leq y_{j}\leq M_{n,r}+r \text{ for } 1\leq j\leq
n-1\},
\]
\[
\text{where } M_{n,r}=\sup\left\{  Q\left(  x-x_{j}\right)  :x\in
D_{Q}(x_{n}, 4r),\ 1\leq j\leq n\right\} .
\]

Let
$b_{n}=b_{n,r,\varepsilon}:\ell_{\infty}^{n}\rightarrow\lbrack0,2]$
be the function defined by
\[
b_{n}(y)=(1+\varepsilon)\max\{0,1-\frac{1}{r}\text{dist}_{\infty}(y,A_{n}^{\prime})\},
\]
where $\text{dist}_{\infty}(y,A)=\inf\{\Vert
y-a\Vert_{\infty}:a\in A\}$.

In the sequel, in order not to be too burdened by notation, will
omit the subscripts $r, \varepsilon$ whenever they do not play a
relevant role locally in the argument.

It is clear that support$(b_{n})=A_{n}$, that
$b_{n}=1+\varepsilon$ on $A_{n}^{\prime}$, and that $b_{n}$ is
$(2/r)$-Lipschitz (note in particular that the Lipschitz constant
of $b_{n}$ does not depend on $n\in\N$ or $\varepsilon\in (0,1)$).

Since the function $b_n=b_{n, r, \varepsilon}$ is
$\frac{2}{r}$-Lipschitz and bounded by $2$ on $\mathbb{R}^{n}$, it
is a standard fact that the normalized integral convolutions of
$b_{n}$ with the Gaussian-like kernels $y\mapsto G_{\kappa }(y):=
e^{-\kappa\sum_{j=1}^{{n}}2^{-j}y_{j}^{2}}$,
\[
x\mapsto\frac{1}{T_{\kappa}}b_{n} * G_{\kappa}(x)=\frac{1}{\int_{\mathbb{R}%
^{n}}e^{-\kappa\sum_{j=1}^{{n}}2^{-j}y_{j}^{2}}dy} \int_{\mathbb{R}^{n}}%
b_{n}(y)e^{-\kappa\sum_{j=1}^{n}2^{-j}(x_{j}-y_{j})^{2}}dy,
\]
\[
\text{where } T_{\kappa}=T_{\kappa, n}=\int_{\mathbb{R}^{n}}e^{-\kappa\sum_{j=1}^{{n}}%
2^{-j}y_{j}^{2}}dy,
\]
converge to $b_{n}$ uniformly on $\mathbb{R}^{n}$ as
$\kappa\to+\infty$. Moreover, since $b_n=b_n'=0$ on the open set
$\ell_{\infty}^n\setminus A_n$, we also have that
    $$
\lim_{\kappa\to+\infty}\left(\frac{1}{T_{\kappa, n}}b_{n, r,
\varepsilon}
* G_{\kappa}\right)' (x)=b_{n,r,\varepsilon}'(x)
    $$
uniformly in $x\in \{y\in\ell_{\infty}^{n} \, : \,
\textrm{dist}_{\infty}(y, A_n)\geq r\}$. Therefore, for each
$n\in\mathbb{N}$ we can find $\kappa_n =\kappa_{n, r,
\varepsilon}>0$ large enough so that
$$
|b_{n, r, \varepsilon}(x)-\frac{1}{T_{\kappa_{n,
\varepsilon}}}b_{n, r, \varepsilon}
* G_{\kappa_{n, \varepsilon}}(x)|\leq\frac{\varepsilon}{2} \,\, \text{ for
all } \,\, x\in\mathbb{R}^{n}\, \textrm{ and }  \eqno(*)
$$
$$
\|b_n'(x)-(b_n*G_{\kappa_n})'(x)\|\leq\frac{\varepsilon}{2} \,\,
\textrm{ whenever } \,\, \textrm{dist}_{\infty}(x, A_n)\geq r.
\eqno(**)
    $$
For reasons that will become apparent later on, we shall also take
each $\kappa_{n, r, \varepsilon}$ to be large enough so that
\[
\kappa_{n}^{n}\geq 2(\sqrt{2})^{n}(n!)^{2}. \eqno (***)
\]
Also, note that when $r\geq 1$ then the functions $b_{n, r,
\varepsilon}$ are all $2$-Lipschitz and therefore $\kappa_{n,r,
\varepsilon}$ can be assumed to be independent of $r$.

Let us also define $\widetilde{T}_n
=\int_{\mathbb{R}^{n}}e^{-\sum_{j=1}^{n}2^{-j} y_{j}^{2}}dy$, and
 \[
T_{n}:= T_{\kappa_n,
n}=\int_{\mathbb{R}^{n}}e^{-\kappa_{n}\sum_{j=1}^{n}2^{-j}y_{j}^{2}}dy
 =\frac{1}{\kappa_{n}^{n/2}}\widetilde{T}_{n},
\]
and observe that by a change of variables
    $$
T_{n}=\int_{\mathbb{R}^{n}}e^{-\kappa_{n}\sum_{j=1}^{n}2^{-j}y_{j}^{2}}dy
 = \int_{\mathbb{R}^{n}}e^{-\kappa_{n}\sum_{j=1}^{n}2^{-j}(A_j-y_{j})^{2}}dy
    $$
for any $A_j \in\R$.

Now define $\nu_{n}:\ell^{n}_{\infty}\to\mathbb{R}$ by
\[
\nu_{n} (x) := \frac{1}{T_{\kappa_{n}}}b_{n} * G_{\kappa_{n}}(x) =
\frac
{1}{T_{n}}\int_{\mathbb{R}^{n}}b_{n}(y)e^{-\kappa_{n}\sum_{j=1}^{n}%
2^{-j}(x_{j}-y_{j})^{2}}dy.
\]

Let us note that
\[
\nu_{n}(x)=\frac{1}{T_{n}}\int_{\mathbb{R}^{n}}b_{n}(y)e^{-\kappa_{n}%
\sum_{j=1}^{n}2^{-j}(x_{j}-y_{j})^{2}}dy= \frac{1}{T_{n}}\int_{\mathbb{R}^{n}%
}b_{n}(x-y)e^{-\kappa_{n}\sum_{j=1}^{n}2^{-j}y_{j}^{2}}dy,
\]
and so
\begin{align*}
\left\vert \nu_{n}\left(  x\right)  -\nu_{n}\left(
x^{\prime}\right) \right\vert  &  =\left\vert
\frac{1}{T_{n}}\int_{\mathbb{R}^{n}}\left( b_{n}\left(  x-y\right)
-b_{n}\left(  x^{\prime}-y\right)  \right)
e^{-\kappa_{n}\sum_{j=1}^{n}2^{-j}y_{j}^{2}}dy\right\vert \\
 &  \leq\frac{1}{T_{n}}\int_{\mathbb{R}^{n}}\left\vert
b_{n}\left(  x-y\right)
-b_{n}\left(  x^{\prime}-y\right)  \right\vert e^{-\kappa_{n}\sum_{j=1}%
^{n}2^{-j}y_{j}^{2}}dy\\
&  \leq\frac{2}{r}\ \left\Vert x-x^{\prime}\right\Vert _{\infty}\frac{1}%
{T_{n}}\int_{\mathbb{R}^{n}}e^{-\kappa_{n}\sum_{j=1}^{n}2^{-j}y_{j}^{2}}dy\\
 &  =\frac{2}{r}\left\Vert x-x^{\prime}\right\Vert _{\infty}.
\end{align*}
Hence, $\nu_{n}$ is $\frac{2}{r}$-Lipschitz. Note also that
$\|\nu_n\|_{\infty}\leq\|b_n\|_{\infty}\leq 2$.

\medskip

\noindent Next, consider the map $\lambda_{n}:X\rightarrow
l_{\infty}^{n}$ given by
\[
\lambda_{n}\left(  x\right)  =\left(  Q\left(  x-x_{1}\right)
,...,Q\left( x-x_{n}\right)  \right)  ,
\]
where $Q$ is the real analytic and Lipschitz separating function
constructed in Lemma \ref{separating function Q lemma}.

\medskip

\noindent Then for $n\geq1$ we define (real) analytic maps
$\varphi _{n}=\varphi_{n,r,\varepsilon}:X\rightarrow\mathbb{R}$ by
\begin{align*}
\varphi_{n}\left(  x\right)   &  =\nu_{n}\left(  \lambda_{n}\left(
x\right) \right)  =\nu_{n}\left(  \left\{  Q\left(  x-x_{j}\right)
\right\}
_{j=1}^{n}\right)  \\
&  =\frac{1}{T_{n}}\int_{\mathbb{R}^{n}}b_{n}(y)e^{-\kappa_{n}\sum_{j=1}%
^{n}2^{-j}\left(  Q\left(  x-x_{j}\right)  -y_{j}\right)  ^{2}}dy\text{.}%
\end{align*}
Since $Q$ is Lipschitz on $X$, we have that
\begin{align*}
\left\vert \varphi_{n}\left(  x\right)  -\varphi_{n}\left(
x^{\prime}\right) \right\vert   =\left\vert \nu_{n}\left(
\lambda_{n}\left(  x\right) \right)  -\nu_{n}\left(
\lambda_{n}\left(  x^{\prime}\right)  \right) \right\vert
\leq\frac{2}{r}\left\Vert \lambda_{n}\left(  x\right)
-\lambda_{n}\left( x^{\prime}\right)  \right\Vert
_{\infty}\\
=\frac{2}{r}\left\Vert \left\{  Q\left(  x-x_{j}\right) -Q\left(
x^{\prime}-x_{j}\right)  \right\}  _{j=1}^{n}\right\Vert _{\infty}
\leq\frac{2}{r}\text{Lip}(Q)\left\Vert x-x^{\prime}\right\Vert
_{X},
\end{align*}
hence the collection $\left\{  \varphi_{n, r, \varepsilon}:
n\in\N\right\} $ is uniformly Lipschitz on $X$, with constant
$\frac{2}{r}\text{Lip}(Q)$.

\noindent We can extend the maps $\varphi_{n, r, \varepsilon}$ to
complex valued maps defined on $W_{\delta_{Q}}$ (see Lemma
\ref{separating function Q lemma}), calling them
$\widetilde{\varphi}_{n, r, \varepsilon}$. Namely (where $x\in X$,
$z\in\widetilde{X}$),
\[
\widetilde{\varphi}_{n}\left(  x+z\right)  =\frac{1}{T_{n}}\int_{\mathbb{R}%
^{n}}b_{n}\left(  y\right)
e^{-\kappa_{n}\sum_{j=1}^{n}2^{-j}\left( \widetilde{Q}\left(
x-x_{j}+z\right)  -y_{j}\right)  ^{2}}dy
\]
\noindent Note that the $\widetilde{\varphi}_{n}$ are well defined
(as the $b_n$ have compact supports) and are holomorphic where
$\widetilde{Q}$ is (namely on $\widetilde{V}$). Hence the above
calculation establishes (1) as $\widetilde{\varphi}_{n}\mid
_{X}=\varphi_{n}.$  Bearing in mind that
$\|\nu_n\|_{\infty}\leq\|b_n\|_{\infty}$, it is also clear that
$0\leq\varphi _{n}(x)\leq 1+\varepsilon$ for all $x\in X,
n\in\mathbb{N}$, which proves (2).

\medskip

Let us show (3). For each fixed $x\in X,$ there exists $m=m_{x,r}$
with $x\in D_{Q}(x_{n_x},3r)$ but with $x\notin D_{Q}(x_i,3r)$ for
$i<m$. Observe that if $r\geq s$ then $x\notin D_{Q}(x_i,3s)$ for
$i<m_{x,r}$, so we necessarily have $m_{x,s}\geq m_{x,r}$ for all
$r\geq s$.

This implies that the point $\left(  Q\left( x-x_{1}\right)
,Q\left( x-x_{2}\right)  ,...,Q\left( x-x_{m_{x,r}}\right) \right)
$ belongs to $A_{m, r}^{\prime}$, where the function $b_{m_{x,r}}$
takes the value $1+\varepsilon$. According to $(\ast)$, we have
\begin{eqnarray*}
& & |1+\varepsilon-\varphi_{m}(x)|=\\
& & |b_{m}(x)-(\frac{1}{T_{\kappa_{n_x}}}b_{m}\ast G_{\kappa_{m}%
})(Q(x-x_{1}),Q(x-x_{2}),...,Q(x-x_{m}))|\leq\varepsilon/2,
\end{eqnarray*}
which yields $\varphi_{m}(x)\geq 1+\varepsilon-\varepsilon/2>1$.

\medskip

Properties (4) and (4') are shown similarly: if $Q(x-x_{n})\geq4r$
then the point $(Q(x-x_{1}), ..., Q(x-x_{n}))$ lies in a region of
$\mathbb{R}^{n}$ where the function $b_{n}$ takes the value $0$,
and $(*)$ immediately gives us $\varphi_n(x)\leq\varepsilon/2$.
And if $Q(x-x_n)\geq 5r$ then $(Q(x-x_{1}), ...,
Q(x-x_{n}))\in\{y\in\ell_{\infty}^{n} : \textrm{dist}_{\infty}(y,
A_n)\geq r\}$, so $(**)$ implies that
$\|\varphi_{n}'(x)\|\leq\varepsilon/2$.

\medskip

We finally show the more delicate properties $(5)$, $(6)$ and
$(7)$. For $x\in X$ and $z\in\widetilde{X}$ with $\|z\|<\delta
_{Q}$, according to Lemma \ref{separating function Q lemma}, we
have
\[
\widetilde{Q}\left(  x-x_{j}+z\right)  =Q\left(  x-x_{j}\right)
+Z_{j},
\]
\[
\text{where } Z_{j}\in\mathbb{C} \text{ with } \left\vert
Z_{j}\right\vert \leq C\left\Vert z\right\Vert _{\widetilde{X}},
\]
$C$ being the Lipschitz constant of $\widetilde{Q}$ on the strip
$W_{\delta_{Q}}$.

Now
\begin{align*}
\left(  \widetilde{Q}\left(  x-x_{j}+z\right)  -y_{j}\right)  ^{2}
&
=\left(  Q\left(  x-x_{j}\right)  -y_{j}+Z_{j}\right)  ^{2}\\
&  =\left(  Q\left(  x-x_{j}\right)  -y_{j}\right)  ^{2}+2\left(
Q\left( x-x_{j}\right)  -y_{j}\right)  Z_{j}+Z_{j}^{2}.
\end{align*}
Hence, for $\left\Vert z\right\Vert _{\widetilde{X}}< \delta_{Q}$
we have
\begin{align*}
&  \operatorname{Re}\left(  \widetilde{Q}\left(  x-x_{j}+z\right)
-y_{j}\right)  ^{2}\\
 &  =\left(  Q\left(  x-x_{j}\right)  -y_{j}\right) ^{2}+2\left(
Q\left(
x-x_{j}\right)  -y_{j}\right)  \operatorname{Re}Z_{j}+\operatorname{Re}%
(Z_{j}^{2})\\
 &=\left(Q(x-x_j)-y_j+ \operatorname{Re}Z_{j} \right)^{2}- \left(
\operatorname{Re}Z_{j} \right)^{2}+\operatorname{Re}(Z_{j}^{2})\\
&\geq \left(Q(x-x_j)-y_j+ \operatorname{Re}Z_{j}
\right)^{2}-2C^2\|z\|_{\widetilde{X}}^{2} \, .
\end{align*}
Next, for each $x\in X$ there exists $n_{x, r}$ so that $Q\left(
x-x_{n_{x, r}}\right)< r$ (note that $n_{x, r}$ can be assumed to
be equal to $n_{x,1}$ in the case $r\geq 1$), and also for $n>
n_{x, r}$ and $y\in A_{n}=\ $ support$\left( b_{n}\right) $ we
have $y_{n_{x, r}}\geq 2r.$ Hence, $y_{n_{x, r}}-Q\left(
x-x_{n_{x, r}}\right)
>2r-r=r.$ Thus, for $\left\Vert z\right\Vert
_{\widetilde{X}}<\min\{r/2C, \delta_{Q}\}$, $n>n_{x, r}$ and $y\in
A_n$ we have
    \begin{align*}
& \left( Q\left( x-x_{n_{x, r}}\right) -y_{n_{x, r}}
+\operatorname{Re} (Z_{n_{x, r}})\right)^{2}= \left( y_{n_{x,
r}}-Q\left( x-x_{n_{x, r}}\right)
-\operatorname{Re} (Z_{n_{x, r}})\right)^{2}\\
& \geq \left( y_{n_{x, r}}-Q\left( x-x_{n_{x, r}}\right)
-C\left\Vert z\right\Vert _{\widetilde{X}}\right)^{2}\geq \left( r
-C\left\Vert
z\right\Vert _{\widetilde{X}}\right)^{2} \\
& \geq (r-\frac{r}{2})^{2}=\frac{r^2}{4}.
    \end{align*}

It follows that for $\left\Vert z\right\Vert
_{\widetilde{X}}<\min\{r/2C, \delta_{Q}\}$, $n>n_{x, r}$, and
$y\in A_n$,
\begin{align*}
&  \sum_{j=1}^{n}2^{-j}\operatorname{Re}\left( \widetilde{Q}\left(
x -x_{j}+z\right)  -y_{j}\right)  ^{2}\\
 &  \geq\sum_{j=1}^{n}2^{-j} \left(Q(x-x_j)-y_j+
\operatorname{Re}Z_{j} \right)^{2}-2C^2 \sum_{j=1}^{n}2^{-j}\|z\|_{\widetilde{X}}^{2} \\
 & =\frac{1}{2}\sum_{j=1}^{n}2^{-j} \left(Q(x-x_j)-y_j+
\operatorname{Re}Z_{j} \right)^{2} +
\frac{1}{2}\sum_{j=1}^{n}2^{-j} \left(Q(x-x_j)-y_j+
\operatorname{Re}Z_{j} \right)^{2} -2C^2
\sum_{j=1}^{n}2^{-j}\|z\|_{\widetilde{X}}^{2}\\
 & \geq \frac{1}{2}\sum_{j=1}^{n}2^{-j} \left(Q(x-x_j)-y_j+
\operatorname{Re}Z_{j} \right)^{2} + \frac{1}{2} 2^{-n_{x,
r}}\frac{r^2}{4}
-2C^{2}\sum_{j=1}^{n}2^{-j}\left\Vert z\right\Vert _{\widetilde{X}}^{2}\\
 &  \geq \frac{1}{2}\sum_{j=1}^{n}2^{-j} \left(Q(x-x_j)-y_j+
\operatorname{Re}Z_{j} \right)^{2}+ 2^{-n_{x,
r}-3}r^2-2C^{2}\left\Vert z\right\Vert _{\widetilde{X}}^{2}.
\end{align*}
Define
    $
\delta_{x,r} =r\sqrt{\frac{2^{-n_{x, r} -4}}{C^2}}, \,\,\,\,\,
a_{x, r}=r^2 2^{-n_{x, r} -4}
    $
for $0<r\leq 1$, and put $a_{x,r}:= a_{x,1}$ and $\delta_{x,r}:=
\delta_{x,1}$ if $r\geq 1$. For every $n>n_{x, r}$, $y\in A_n$,
and $z\in\widetilde{X}$ with $\|z\|_{\widetilde{X}}\leq\delta_{x,
r}$, we have
    \begin{align*}
& \sum_{j=1}^{n}2^{-j}\operatorname{Re}\left( \widetilde{Q}\left(
x -x_{j}+z\right)  -y_{j}\right)  ^{2}\\
&\geq \frac{1}{2}\sum_{j=1}^{n}2^{-j} \left(Q(x-x_j)-y_j+
\operatorname{Re}Z_{j} \right)^{2}+
2^{-n_{x, r}-3}r^2-2C^{2}\delta_{x, r}^{2}\\
&=\frac{1}{2}\sum_{j=1}^{n}2^{-j} \left(Q(x-x_j)-y_j+
\operatorname{Re}Z_{j} \right)^{2}+ a_{x, r}.
    \end{align*}
Therefore, for every $n>n_{x, r}$ and $z\in\widetilde{X}$ with
$\|z\|_{\widetilde{X}}\leq\delta_{x, r}$ we can estimate
\begin{align*}
\left\vert \widetilde{\varphi}_{n, r, \varepsilon}\left(
x+z\right) \right\vert & =\left\vert
\frac{1}{T_{n}}\int_{\mathbb{R}^{n}}b_{n}\left( y\right)
e^{-\kappa_{n}\sum_{j=1}^{n}2^{-j}\left(  \widetilde{Q}\left(  x_{0}%
-x_{j}+z\right)  -y_{j}\right)  ^{2}}dy\right\vert \\
 &  =\frac{1}{T_{n}}\int_{\mathbb{R}^{n}}b_{n}\left(  y\right)
e^{-\kappa _{n}\operatorname{Re}\sum_{j=1}^{n}2^{-j}\left(
\widetilde{Q}\left(
x_{0}-x_{j}+z\right)  -y_{j}\right)  ^{2}}dy\\
 & \leq\frac{2}{T_{n}}\int_{A_{n,r}}e^{-\kappa_{n}a_{x, r}
-\frac{1}{2}\kappa_n \sum_{j=1}^{n} 2^{-j}\left( Q(x-x_j)+
\operatorname{Re}(Z_j)-y_j\right)^{2} }dy\\
&\leq\frac{2 e^{-\kappa_n a_{x, r}}}{T_{n}}\int_{\R^n}e^{
-\frac{1}{2}\kappa_n \sum_{j=1}^{n} 2^{-j}\left( Q(x-x_j)+
\operatorname{Re}(Z_j)-y_j\right)^{2} }dy\\
&= \frac{2 e^{-\kappa_n a_{x, r}}}{T_{n}}\int_{\R^n}e^{ -\kappa_n
\sum_{j=1}^{n} 2^{-j}u_{j}^{2} }\left(\sqrt{2}\right)^{n}du \\
&= 2 (\sqrt{2})^n e^{-\kappa_n a_{x, r}}.
\end{align*}

\noindent Now, by choice of $\kappa_n =\kappa_{n, r, \varepsilon}$
(see $(\ast\ast\ast)$ above), we have
    $$
2 (\sqrt{2})^n e^{-\kappa_n a_{x, r}}\leq \frac{2 (\sqrt{2})^n
n!}{\kappa_{n}^{n}a_{x, r}^{n}}\leq\frac{1}{a_{x, r}^{n} n! M_n}.
    $$
Hence
$$|\widetilde{\varphi}_{n, r, \varepsilon}\left(  x+z\right)|\leq \frac
{1}{n!a_{x, r}^{n}M_n}$$ for all $n> n_{x, r}$, $\varepsilon\in
(0,1)$, and $z\in\widetilde{X}$ with $\|z\|\leq\delta_{x, r}$.
This establishes $(5)$.

In particular by taking $n_{x, r,\varepsilon}> n_{x, r}$
sufficiently large, we can guarantee that for $n>n_{x, r,
\varepsilon}$ and $\left\Vert z\right\Vert
_{\widetilde{X}}<\delta_{x, r}$ we have $\left\vert
\widetilde{\varphi}_{n}\left( x_{0}+z\right) \right\vert
<\varepsilon$, which proves $(6)$.

As for $(7)$, we can write, as above, $ \widetilde{Q}\left(
x-x_{j}+z\right)  =Q\left(  x-x_{j}\right) +Z_{j}$, where
$Z_{j}\in\mathbb{C}$ with $\left\vert Z_{j}\right\vert \leq
C\left\Vert z\right\Vert _{\widetilde{X}}\leq C\delta_{x, r}$, and
we have
$$
\operatorname{Re}\left(  \widetilde{Q}\left(  x-x_{j}+z\right)
-y_{j}\right)  ^{2}\geq \left(Q(x-x_j)-y_j+ \operatorname{Re}Z_{j}
\right)^{2}-2C^2\delta_{x,r}^{2},
$$
hence
\begin{align*}
\left\vert \widetilde{\varphi}_{n, r, \varepsilon}\left(
x+z\right) \right\vert &
=\frac{1}{T_{n}}\int_{\mathbb{R}^{n}}b_{n}\left(  y\right)
e^{-\kappa _{n}\operatorname{Re}\sum_{j=1}^{n}2^{-j}\left(
\widetilde{Q}\left(
x_{0}-x_{j}+z\right)  -y_{j}\right)  ^{2}}dy\\
& \\
& \leq\frac{1+\varepsilon}{T_{n}}\int_{\R^n}e^{-\kappa_{n}
\sum_{j=1}^{n}2^{-j}\left[\left(Q(x-x_j)+ \operatorname{Re}Z_{j}
-y_j
\right)^{2}-2C^2\delta_{x,r}^{2}\right]}dy\\
& =\frac{(1+\varepsilon)e^{2\kappa_n
C^2\delta_{x,r}^{2}}}{T_{n}}\int_{\R^n}e^{-\kappa_{n}
\sum_{j=1}^{n}2^{-j}
\left(Q(x-x_j)+ \operatorname{Re}Z_{j} -y_j \right)^{2}}dy\\
& = (1+\varepsilon)e^{2\kappa_n C^2\delta_{x,r}^{2}}.
\end{align*}
This is true for every $n\in\N$, $z\in\widetilde{X}$ with
$\|z\|_{\widetilde{X}}\leq \delta_{x, r}$. On the other hand, for
$n>n_{x, r}$ we already know that
    $$
\left\vert \widetilde{\varphi}_{n, r, \varepsilon}\left(
x+z\right) \right\vert \leq \frac{1}{a_{x,r} n!}
    $$
if $\|z\|_{\widetilde{X}}\leq \delta_{x, r}$, and since this
sequence converges to $0$ we can find $m_{x, r, \varepsilon}$ such
that
$$
\left\vert \widetilde{\varphi}_{n, r, \varepsilon}\left(
x+z\right) \right\vert \leq \frac{1}{a_{x,r} n!}<1+2\varepsilon
\textrm{ for } n\geq m_{x, r, \varepsilon}, \,
\|z\|_{\widetilde{W}}\leq\delta_{x, r}.
    $$
Then we can define
$$
\delta_{x, r, \varepsilon}= \min\left \{ \delta_{x, r},
\sqrt{\frac{1}{2\kappa_{1, r, \varepsilon} C^{2}}\log
\left(\frac{1+2\varepsilon}{1+\varepsilon}\right)}, ... ,
\sqrt{\frac{1}{2\kappa_{m_{x, r, \varepsilon}, r, \varepsilon}
C^{2}}\log \left(\frac{1+2\varepsilon}{1+\varepsilon}\right)}
\right \},
$$
so we get
$$
\left\vert \widetilde{\varphi}_{n, r, \varepsilon}\left(
x+z\right) \right\vert \leq 1+2\varepsilon
    $$
if $\|z\|_{\widetilde{X}}\leq \delta_{x, r, \varepsilon}$, for all
$n\in\N$ and $r>0$. In the case $r\geq 1$ the numbers $\kappa_{n,
r, \varepsilon}, \delta_{x,r}, a_{x,r}$ are independent of $r$,
hence so are $m_{x, r, \varepsilon}$ and $\delta_{x, r,
\varepsilon}$.
\end{proof}

\begin{remark}
{\em In this paper we will only use the above Lemma in the case
$r\geq 1$, and we will not need property $(4')$. The $M_n$ can
also be dropped from property $(5)$. However, we shall use the
full power of this Lemma in \cite{AFK3}.}
\end{remark}

\bigskip

\subsection{A refinement of the implicit function theorem for
holomorphic mappings on complex Banach spaces.}

We shall use the following refinement of \cite[Theorem 1.1]{CHP},
which provides a lower bound on the size of the domains of
existence of the solutions to a family of complex analytic
implicit equations, in a uniform manner with respect to the
equation, whenever appropriate bounds are available for the
functions defining the equations and its first derivatives with
respect to the dependent variable.

\medskip

\begin{proposition}\label{refinement of CHP}
Let $\psi :E\times\C\times P\to\C$ be a function, where $E$ is a
complex Banach space and $P$ is a set of parameters. Let $(x_0,
\mu^0)\in E\times\C$, $R>0$, and suppose that for each $p\in P$
the function
$$(x, \mu)\mapsto \psi_{p} (x, \mu)=\psi(x, \mu,
p)$$ is (complex) analytic on $B=\{(x, \mu) : \|x-x_0\|\leq R, \,
|\mu-\mu^0|\leq R\}$, and that there exist $M, A>0$ such that
    $$
\left|\frac{\partial\psi_{p}}{\partial\mu}(x_0, \mu_0)\right|\geq
A, \,\,\, \textrm{ and } \,\, |\psi_{p}(x,\mu)|\leq M \,\,
\textrm{ for all } (x, \mu)\in B, p\in P.
    $$
Then there exist
    $
0<r=r(A,R,M)<R$ and $ s=s(A,R,M)>0
    $
so that for each $p\in P$ there exists a unique analytic function
$\mu=\mu_{p}(x)$, defined on the ball $B(x_0, s)$, such that
$\psi_p(x, \mu)=0$ for $(x, \mu)\in B(x_0, s)\times B(\mu^0, r)$
if and only if $\mu=\mu_p(x)$ for $x\in B(x_0, s)$.
\end{proposition}
\begin{proof}
This is essentially a restatement of Theorem 1.1 in \cite{CHP},
where the result is proved (with estimates for $r$ and $s$) in the
case when $P$ is a singleton and $E=\C^n$. The part of the proof
in \cite{CHP} that does not follow directly from the implicit
function theorem for holomorphic mappings combines an estimation
of the Taylor series of $\psi_{p}$ at $(x_0, \mu_0)$ with respect
to the variable $\mu\in\C$, and an application of Rouche's theorem
to the functions $\C\ni \mu\mapsto \psi_{p}(x_0, \mu)\in\C$ and
$\C\ni\mu\mapsto \psi_{p}(x, \mu)\in\C$. In these two steps the
variable $x\in E$ can be regarded as a parameter, so the same
argument applies, with obvious changes, to the present situation.
\end{proof}

\medskip

\subsection{Proof of the main result in the case of a $1$-Lipschitz function taking values in $[0,1]$.}

Let us define
$$A=\{\alpha=\{\alpha_n\}_{n=1}^{\infty}\in\ell_{\infty} \,\, : \,\,
1 \leq \alpha_n\leq 1001 \, \textrm{ for all } n\in\N\}.$$ For
every $\alpha\in A$ and $z=\{z_n\}_{n=1}^{\infty}\in
\widetilde{c_{0}}\setminus\{0\}$, let us denote $\alpha
z:=\{\alpha_n z_n \}_{n=1}^{\infty}$, and observe that $\alpha
z\in \widetilde{c_{0}}\setminus\{0\}$. In fact the mapping
$\alpha: \widetilde{c_0}\to \widetilde{c_0}$, defined by
$\alpha(z)=\alpha z$ is a linear isomorphism satisfying
$\|z\|_{\infty}\leq \|\alpha z\|_{\infty}\leq 1001\|z\|_{\infty}$.

\begin{lemma}\label{holomorphic extension of the composition of lambda with Phi}
Let $r\geq 1$, and let $\varphi_n =\varphi_{n, r, \varepsilon} ,
\widetilde{\varphi}_n =\widetilde{\varphi}_{n, r, \varepsilon}$ be
the collections of functions defined in Lemma \ref{construction of
the sup partition of unity}. We have that:
\begin{enumerate}
\item There exists an open
neighborhood $\widetilde{V}_{1}\subset\widetilde{V}$ of $X$ in
$\widetilde{X}$ such that the mapping $\widetilde{\Phi}_{\alpha,
r,
\varepsilon}=\widetilde{\Phi}:\widetilde{V}_{1}\to\widetilde{c_0}$
defined by
    $$
\widetilde{\Phi}(z)=\{\alpha_n\widetilde{\varphi}_{n}(z)\}_{n=1}^{\infty}
    $$
is holomorphic, for any $\alpha\in A$, $r\geq 1$, $\varepsilon \in
(0,1)$.
\item  Fix $\varepsilon\in (0, 1)$. Then there exists an open neighborhood $\widetilde{W}_{\varepsilon}$
of $X$ in $\widetilde{X}$ such that for every $\alpha\in A$,
$r\geq 1$, there is a holomorphic extension
$\widetilde{F}_{\alpha, r, \varepsilon}$ of
$\lambda\circ\Phi_{\alpha, r, \varepsilon}$, defined on
$\widetilde{W}_{\varepsilon}$, where $\lambda$ is the Preiss norm
from section 4.1, and such that
    $$
|\widetilde{F}_{\alpha, r, \varepsilon}(x+iy)-
\lambda\circ\Phi_{\alpha, r, \varepsilon}(x)| \leq\varepsilon
    $$
for every $x, y\in X$ with $x+iy\in \widetilde{W}_\varepsilon$.
\end{enumerate}
\end{lemma}
\begin{proof}
Since $r\geq 1$ we have $a_{x,r}=a_{x,1}$ and
$\delta_{x,r}=\delta_{x,1}$ in Lemma \ref{construction of the sup
partition of unity}, so we will simply denote these numbers by
$a_x$ and $\delta_x$, respectively. Because the numerical series
$\sum_{n=1}^{\infty}1/n!a_{x}^{n}$ is convergent, and bearing in
mind property $(5)$ of Lemma \ref{construction of the sup
partition of unity}, we have that the series of functions
    $$
\sum_{n=1}^{\infty}|\alpha_n \widetilde{\varphi}_{n}\left(
z\right)|
    $$
is uniformly convergent on the ball $B_{\widetilde{X}}(x,
\delta_{x})$, which clearly implies that the series
    $$
\sum_{n=1}^{\infty}\alpha_n\widetilde{\varphi}_{n}(z)e_n
    $$
is uniformly convergent for $z\in B_{\widetilde{X}}(x,
\delta_{x})$, for any $\{a_n\}_{n=1}^{\infty}\in\ell_{\infty}$. If
we set $\widetilde{V}_1=\bigcup_{x\in X}B_{\widetilde{X}}(x,
\delta_{x})$ and note that for each $n\in\N$ the mapping
$$\widetilde{V}\ni z\mapsto \varphi_n (z)e_n
\in\widetilde{c_0}$$ is holomorphic, it is clear that
$\widetilde{\Phi}_{\alpha, r,
\varepsilon}=\sum_{n=1}^{\infty}\widetilde{\varphi}_{n, r,
\varepsilon} e_n$, being a series of holomorphic mappings which
converges locally uniformly in $\widetilde{V}_1$, defines a
holomorphic mapping from $\widetilde{V}_1$ into $\widetilde{c_0}$.
Note that $\widetilde{V}_1$ is independent of $\alpha, r,
\varepsilon$ because so is $\delta_x$.

\medskip

Let us now prove $(2)$. Define the function $\psi_{\alpha, r,
\varepsilon}: \widetilde{V}_1 \times\C\setminus\{0\} \to \C$ by
    $$
\widetilde{\psi}_{\alpha, r, \varepsilon}(w,
\mu)=\sum_{n=0}^{\infty}\left( \frac{\alpha_n
\widetilde{\varphi}_{n, r, \varepsilon}(w)}{\mu}\right)^{2n}-1.
    $$
The function $\widetilde{\psi}_{\alpha, r, \varepsilon}$ is
holomorphic on $\widetilde{V}_1$, because it is the composition of
$\widetilde{\Phi}_{\alpha, r, \varepsilon}:\widetilde{V}_1
\to\widetilde{c}_0$ with the function
$\widetilde{C}:\widetilde{c}_0\times\C\setminus\{0\}\to\C$ defined
by $\widetilde{C}(\{z_n\})=\sum_{n=1}^{\infty} z_{n}^{2n}-1$,
which is also holomorphic on
$\widetilde{c}_0\times\C\setminus\{0\}$.

For $\varepsilon>0$ fixed and $x\in X$ given, we  are going to
apply the preceding Proposition to the family of functions
$\widetilde{\psi}_{\alpha, r, \varepsilon}$ indexed by
$P=\{(\alpha, r) \, : \, \alpha\in A, r\geq 1\}$. So we have to
find appropriate bounds for $\widetilde{\psi}_{\alpha, r,
\varepsilon}$ and $\partial{\psi}_{\alpha, r,
\varepsilon}/\partial \mu$.

For our given $x\in X$, we can apply properties $(5)$ and $(7)$ of
Lemma \ref{construction of the sup partition of unity} to find
$n_x, a_x, \delta_x>0$ (independent of $r, \varepsilon$) and
$\delta_{x, \varepsilon}>0$ (independent of $r$) such that
    $$
|\widetilde {\varphi}_{n, r, \varepsilon} (w)|<\frac{1}{n!
a_x^{n}} \,\, \textrm{ for } \,  n>n_{x}, \, r\geq 1,  \, w\in
\widetilde{X} \textrm{ with } \|w-x\|_{\widetilde{X}}<\delta_x.
    $$
and
    $$
|\widetilde {\varphi}_{n, r, \varepsilon} (w)|\leq 1+2\varepsilon
\,\, \textrm{ for } n\in\N, r\geq 1, \, w\in \widetilde{X}
\textrm{ with }\|w-x\|_{\widetilde{X}}\leq \delta_{x,
\varepsilon}.
    $$
Let us fix $0<R_{x, \varepsilon} <\min\{ \varepsilon, 1/6,
\delta_x, \delta_{x, \varepsilon}\}$, and observe that
$\widetilde{V}_1 \supset B_{\widetilde{X}}(x, R_{x,
\varepsilon})$.

Note that $\lambda(\Phi_{\alpha, r, \varepsilon}(x))\geq
\|\{\alpha_n\varphi_{n}(x)\}_{n=1}^{\infty}\|_{\infty}>1$ by
property $(3)$ of Lemma \ref{construction of the sup partition of
unity}, so in particular $|\mu|\geq \lambda(\Phi_{\alpha, r,
\varepsilon}(x))-R_{x, \varepsilon}\geq 1-1/6=5/6$ whenever $|\mu
-\lambda(\Phi_{\alpha, r, \varepsilon}(x))|\leq R_{x,
\varepsilon}$.

Then we have, for every $(\alpha, r)\in P$, for $
\|w-x\|_{\widetilde{X}}\leq R_{x, \varepsilon}$, and for $|\mu
-\lambda(\Phi_{\alpha, r, \varepsilon}(x))|\leq R_{x,
\varepsilon}$,
    \begin{align*}
    & |\widetilde{\psi}_{\alpha, r, \varepsilon}(w, \mu)|\leq
    \left | \sum_{n=1}^{\infty} \left( \frac{\alpha_n
\widetilde{\varphi}_{n, r, \varepsilon}(w)}{\mu}\right)^{2n}-1
\right|\\
& \leq 1+
\sum_{n=1}^{n_x}\left(\frac{1001(1+2\varepsilon)}{5/6}\right)^{2n}
+\sum_{n=n_x +1}^{\infty}\left(\frac{1}{n! a_{x}^{n}}\right)^{2n}
:= M_{x, \varepsilon}
    \end{align*}
(observe that $M_{x, \varepsilon}$ is independent of $(r,
\alpha)\in P$).

On the other hand we can apply property $(3)$ of Lemma
\ref{construction of the sup partition of unity} to find, for
every $r\geq 1$, numbers
$$m_{x,r}\leq m_{x,1} \, \textrm{ such that } \, \varphi_{m_{x,
r}, r, \varepsilon}(x)>1.$$ Since $\lambda\leq
2\|\cdot\|_{\infty}$ on $c_0$, and using property $(2)$ of Lemma
\ref{construction of the sup partition of unity}, we have
    $$
\lambda(\Phi_{\alpha, r, \varepsilon}(x))\leq 2\|\Phi_{\alpha, r,
\varepsilon}(x)\|_{\infty}\leq 2002\|\Phi_{1, r,
\varepsilon}(x)\|_{\infty}\leq 4004.
    $$
So, if we set $$A_{x}:=\frac{1}{4004^{2m_{x,1}+1}}$$ (clearly
independent of $(r, \alpha)\in P$), and we use the fact that
$m_{x, r}\leq m_{x,1}$ for all $r\geq 1$, we can estimate
    \begin{align*}
    & \left | \frac{\partial \psi_{\alpha, r, \varepsilon}}{\partial\mu}
    (x, \lambda(\Phi_{\alpha, r, \varepsilon} (x))) \right |
    = \left | -\sum_{n=1}^{\infty}2n \left( \frac{\alpha_n \varphi_n (x)}{\lambda(\Phi_{\alpha, r, \varepsilon}(x))}
    \right)^{2n} \cdot \frac{1}{\lambda(\Phi_{\alpha, r,
    \varepsilon}(x))}\right|\\
    & \geq \sum_{n=1}^{\infty} 2n \left( \frac{\varphi_n
    (x)}{4004}
    \right)^{2n} \cdot \frac{1}{4004}\geq
    \frac{2m_{x,r}}{4004}\left(\frac{\varphi_{m_{x,r}}(x)}{4004}\right)^{2m_{x,r}}\\
    & \geq \frac{m_{x,r}}{4004} \left(\frac{1}{4004}\right)^{2m_{x,r}}\geq
    \frac{1}{4004}\left(\frac{1}{4004} \right)^{2m_{x,r}}\geq
    \frac{1}{4004}\left( \frac{1}{4004}\right)^{2m_{x,1}}=A_{x}.
    \end{align*}

Therefore, according to the preceding Proposition, we can find a
number $s_x= s_{x, \varepsilon}>0$ (independent of $\alpha$, $r$)
such that there is a holomorphic solution
$\widetilde{\mu}=\widetilde{F}_{\alpha, r, \varepsilon}(w)$ to the
implicit equation $\widetilde{\psi}_{\alpha, r, \varepsilon}(w,
\mu)=0$, defined on the ball $B_{\widetilde{X}}(x, s_{x})$, for
every $\alpha\in A, r\geq 1$. Since the solution is locally
unique, two holomorphic functions which coincide on a neighborhood
of a point are equal to one another in the connected component of
that point, and the function $w\mapsto
\widetilde{\lambda}(\widetilde{\Phi}_{\alpha, r, \varepsilon}(w))$
also solves the implicit equation $\widetilde{\psi}_{\alpha, r,
\varepsilon}(w, \mu)=0$ for $\mu$ in terms of $w$ in a
neighborhood of $x$, one easily deduces that
$\widetilde{F}_{\alpha, r, \varepsilon}$ can be defined on all of
    $$
\widetilde{W}_{\varepsilon}:=\bigcup_{x\in X}B_{\widetilde{X}}(x,
s_{x, \varepsilon}),
    $$
and that $\widetilde{F}_{\alpha, r, \varepsilon}$ is a holomorphic
extension of $\lambda\circ\Phi_{\alpha, r, \varepsilon}$.

Besides, by the same Proposition, we also have that for every
$\alpha\in A, r\geq 1$ the function $\widetilde{F}_{\alpha, r,
\varepsilon}$ maps the ball $B_{\widetilde{X}}(x, s_{x})$ into a
disc $B_{\C}(\lambda(\Phi_{\alpha, r, \varepsilon}(x)), r_x)$,
where $0<r_x=r_{x, \varepsilon}< R_{x, \varepsilon}\leq
\varepsilon$. Hence we have that $$ |\widetilde{F}_{\alpha, r,
\varepsilon}(x+iy)- \lambda\circ\Phi_{\alpha, r, \varepsilon}(x)|
\leq\varepsilon
    $$
for every $x, y\in X$ with $x+iy\in \widetilde{W}_\varepsilon$.
\end{proof}

\bigskip

Now we turn to the proof of Theorem \ref{analytic Lip
approximation}. We shall first consider the case when $f:X\to [1,
1001]$ is $L$-Lipschitz with $0< L\leq 1$. Define $r=1/L$, so
$r\geq 1$, rename $\varepsilon=\delta \in (0, 10^{-4})$, and apply
Lemma \ref{construction of the sup partition of unity} to find a
corresponding collection of functions $\varphi_n=\varphi_{n,
\frac{1}{L}, \delta}$, $n\in\N$. The functions $\varphi_n$ are
thus $2C/r$-Lipschitz, where we denote $C=\textrm{Lip}(Q)$.

Define a function $g:X\to\mathbb{R}$ by
    $$
g(x)=\frac{\lambda(\{f(x_n)\varphi_{n}(x)\}_{n=1}^{\infty})}{\lambda(\{\varphi_{n}(x)\}_{n=1}^{\infty})}=
\frac{\lambda\circ\Phi_{\alpha, r,
\delta}(x)}{\lambda\circ\Phi_{1, r, \delta}(x)}
    $$
where here we denote $\alpha=\{f(x_n)\}_{n=1}^{\infty}\in A$, and
also $1=(1, 1, 1, ...)\in A\subset \ell_{\infty}$.

The function $g$ is well defined, is real analytic, and has a
holomorphic extension $\widetilde{g}$ defined on the neighborhood
$\widetilde{U}:=\widetilde{W}_{\delta}$ of $X$ in $\widetilde{X}$
provided by Lemma \ref{holomorphic extension of the composition of
lambda with Phi}, given by
    $$
\widetilde{g}(w)=\frac{\widetilde{F}_{\alpha, r,
\delta}(w)}{\widetilde{F}_{1, r, \delta}(w)}.
    $$

Since the functions $\varphi_n$ are $2C/r$-Lipschitz, the norm
$\lambda:c_0\to\R$ is $2$-Lipschitz (with respect to the usual
norm $\|\cdot\|_{\infty}$ of $c_0$), and $1\leq f\leq 1001$, we
have
\begin{align*}
&
|\lambda(\{f(x_n)\varphi_{n}(x)\}_{n=1}^{\infty})-\lambda(\{f(x_n)\varphi_{n}(y)\}_{n=1}^{\infty})|\leq
\\
&  2\| \{f(x_n)(\varphi_{n}(x)-\varphi_{n}(y))\}_{n=1}^{\infty})
\|_{\infty}\leq 2002\|
\{\varphi_{n}(x)-\varphi_{n}(y)\}_{n=1}^{\infty})
\|_{\infty}\leq\\
&  \frac{4004C}{r}\|x-y\|,
\end{align*}
that is the function $\lambda\circ\Phi_{\alpha, r, \delta}$ is
$4004C/r$-Lipschitz on $X$, and is bounded by $2002$. Similarly,
since the function $t\mapsto 1/t$ is $1$-Lipschitz on $[1,
\infty)$ and $\lambda\circ\Phi_{1, r, \delta}$ is bounded below by
$1$, we have that the function $1/\lambda\circ\Phi_{1, r, \delta}$
is also $4004C/r$-Lipschitz on $X$, and bounded above by $1$.
Therefore the product satisfies
    $$
\textrm{Lip}(g)\leq 2002\times
\frac{4004C}{r}+1\times\frac{4004C}{r}=\frac{8020012 C}{r}=8020012
C\textrm{Lip}(f).
    $$
On the other hand, we have
\begin{align*}
\left\vert g(x)-f(x)\right\vert  &  =\left\vert
\frac{\lambda\left( \left\{ f\left(  x_{j}\right)
\varphi_{j}\left(  x\right) \right\}  \right) }{\lambda\left(
\left\{  \varphi_{j}\left( x\right)  \right\}  \right)
}-f\left(  x\right)  \right\vert \\
& \\
&  =\left\vert \frac{\lambda\left(  \left\{  f\left(  x_{j}\right)
\varphi_{j}\left(  x\right)  \right\}  \right)  }{\lambda\left(
\left\{ \varphi_{j}\left(  x\right)  \right\}  \right)
}-\frac{f\left(  x\right) \lambda\left(  \left\{ \varphi_{j}\left(
x\right)  \right\}  \right) }{\lambda\left( \left\{
\varphi_{j}\left(  x\right)  \right\}  \right)
}\right\vert \\
& \\
&  =\frac{1}{\lambda\left(  \left\{  \varphi_{j}\left(  x\right)
\right\} \right)  }|\lambda\left(  \left\{  f\left(  x_{j}\right)
\varphi_{j}\left( x\right)  \right\}  \right)  -\lambda\left(
\left\{  f\left(  x\right)
\varphi_{j}\left(  x\right)  \right\}  \right)  |\\
& \\
&  \leq\frac{1}{\lambda\left(  \left\{  \varphi_{j}\left( x\right)
\right\} \right)  }\lambda\left(  \left\{
f(x_{j})\varphi_{j}\left(  x\right)
\right\}  -\{f\left(  x\right)  \varphi_{j}\left(  x\right)  \}\right) \\
& \\
&  =\frac{1}{\lambda\left(  \left\{  \varphi_{j}\left(  x\right)
\right\} \right)  }\lambda\left(  \{\varphi_{j}\left(  x\right)
\left(  f\left(
x_{j}\right)  -f\left(  x\right)  \right)  \}\right)  \text{.}%
\end{align*}
Now recall that
\begin{align*}
\lambda\left(  \{\varphi_{j}\left(  x\right)  \left(  f\left(
x_{j}\right) -f\left(  x\right)  \right)  \}\right)   &
\leq2\left\Vert \{\varphi _{j}\left(  x\right)  \left(  f\left(
x_{j}\right)  -f\left(  x\right)
\right)  \}\right\Vert _{\infty}\\
& \\
&  =2\max_{j}\left\{  \varphi_{j}\left(  x\right)  \left\vert
f\left(
x_{j}\right)  -f\left(  x\right)  \right\vert \right\}  \text{.}%
\end{align*}
Set $J=\{j:x\in D_{Q}(x_{j},4r)\}$. For $j\in J$, according to
Lemma \ref{separating function Q lemma} we have $\left\Vert
x-x_{j}\right\Vert _{X}<8r$ and so, because $\textrm{Lip}(f)=1/r$,
\[
\varphi_{j}\left(  x\right)  \left\vert f\left(  x_{j}\right)
-f\left( x\right)  \right\vert \leq \varphi_{j}\left(  x\right) 8.
\]
It follows that for $j\in J$%
\[
\frac{\varphi_{j}\left(  x\right)  \left\vert f\left( x_{j}\right)
-f\left( x\right)  \right\vert }{\lambda\left( \left\{
\varphi_{j}\left(  x\right) \right\}  \right)
}\leq\frac{\varphi_{j}\left(  x\right)  8 }{\left\Vert \left\{
\varphi_{j}\left(  x\right)  \right\}  \right\Vert _{\infty}}\leq
8.
\]
On the other hand, for $j\notin J$ we have by part (4) of Lemma
\ref{construction of the sup partition of unity}, that
\[
\varphi_{j}\left(  x\right)  \left\vert f\left(  x_{j}\right)
-f\left( x\right)  \right\vert \leq 2002\varphi_{j}\left( x\right)
\leq 2002 \delta<1.
\]
Hence, given that $\lambda\left(  \left\{  \varphi_{j}\left(
x\right) \right\}  \right)  \geq1$, we have for $j\notin J$
\[
\frac{\varphi_{j}\left(  x\right)  \left\vert f\left( x_{j}\right)
-f\left( x\right)  \right\vert }{\lambda\left( \left\{
\varphi_{j}\left(  x\right) \right\}  \right)  }\leq 1.
\]
It follows that
\[
\left\vert g(x)-f(x)\right\vert \leq 8.
\]

If we reset $C$ to $8020012 C$, this argument proves Theorem
\ref{analytic Lip approximation} in the case when $\varepsilon
=8$, $f:X\to [1, 1001]$, $\textrm{Lip}(f)\leq 1$.

Moreover, according to property $(2)$ of Lemma \ref{holomorphic
extension of the composition of lambda with Phi}, we have
$$
|\widetilde{F}_{\alpha, r, \delta}(x+iy)- \lambda(\Phi_{\alpha, r,
\delta}(x))| \leq\delta
    $$
for every $x, y\in X$ with $x+iy\in \widetilde{W}_\delta$,
$\alpha\in A$. Therefore, taking into account that
$\lambda(\Phi_{\alpha, r, \delta}(x))\leq 2002$ and
$1\leq\lambda(\Phi_{1, r, \delta}(x))\leq 2$, we have
\begin{eqnarray*}
& &|\widetilde{g}(x+iy)-g(x)|=\left | \frac{\widetilde{F}_{\alpha,
r, \delta}(x+iy)}{\widetilde{F}_{1, r, \delta}(x+iy)}-
\frac{\lambda (\Phi_{\alpha, r, \delta}(x))}{\lambda (\Phi_{1, r,
\delta}(x))}
\right | =\\
& & \left |\frac{1}{\widetilde{F}_{1, r,
\delta}(x+iy)\lambda(\Phi_{1, r, \delta}(x))}  \right | \cdot  |
\lambda(\Phi_{1,r, \delta}(x)) \left(\widetilde{F}_{\alpha, r,
\delta}(x+iy)- \lambda(\Phi_{\alpha, r, \delta}(x))\right)
\\ & & +\lambda(\Phi_{\alpha, r, \delta}(x))\left(\lambda(\Phi_{1, r,
\delta}(x))-\widetilde{F}_{1, r, \delta}(x+iy)\right) |  \leq \\
& & \frac{1}{1-\delta}\left( 2\delta
+2002\delta\right)=\frac{\delta}{1-\delta}2004
\end{eqnarray*}
for every $x, y\in X$ with $x+iy\in \widetilde{W}_\delta$,
$\alpha\in A$, $r\geq 1$.

\medskip

Up to scaling and subtracting appropriate constants we have thus
proved the following intermediate result.
\begin{proposition}\label{intermediate result}
Let $X$ be a Banach space having a separating polynomial, and let
$\eta\in (0, 1/2)$. Then there exists $C_{\eta}\geq 1$ (depending
only on $X$ and $\eta$) such that, for every $\delta>0$ there is
an open neighborhood $\widetilde{U}_{\delta, \eta}$ of $X$ in
$\widetilde{X}$ such that, for every Lipschitz function
$f:X\to[0,1]$ with $\textrm{Lip}(f)\leq 1$, there exists a real
analytic function $g:X\to\mathbb{R}$, with holomorphic extension
$\widetilde {g}:\widetilde{U}_{\delta, \eta}\to\mathbb{C}$, such
that
\begin{enumerate}
\item $|f(x)-g(x)|\leq \eta$ for all $x\in X$.
\item $g$ is Lipschitz, with $\textrm{Lip}(g)\leq
C_\eta \textrm{Lip}(f)$.
\item $|\widetilde{g}(x+iy)-g(x)|\leq \delta$ for all $z=x+iy\in\widetilde{U}_{\delta, \eta}$.
\end{enumerate}
\end{proposition}

\medskip

Next we are going to see that this result remains true for all
bounded, nonnegative functions $f$ with $\textrm{Lip}(f)\leq 1$ if
we allow $C$ to be slightly larger and we replace $\eta$ by $1$.

For a $1$-Lipschitz, bounded function $f:X\rightarrow [  0,
+\infty)$ we
 define, for $n\in\N$, the functions
\[
f_{n}(x)=
\begin{cases}
f(x)-n+1 & \text{ if }n-1\leq f(x)\leq n,\\
0 & \text{ if }f(x)\leq n-1,\\
1 & \text{ if }n\leq f(x).
\end{cases}
\]
Note that since $f$ is bounded there exists $N\in\N$ such that
$f_n=0$ for all $n\geq N$. The functions $f_{n}$ are clearly
$1$-Lipschitz and take values in the interval $[0,1]$, so, for any
given $\delta, \eta>0$ (to be fixed later on), by what has been
proved immediately above, there exist $C_{\eta}\geq 1$, a
neighborhood $\widetilde{U}_{\delta, \eta}$ of $X$ in
$\widetilde{X}$ and $C_{\eta}$-Lipschitz, real analytic functions
$g_{n}:X\rightarrow\mathbb{R}$, with holomorphic extensions
$\widetilde{g}_n :\widetilde{U}_{\delta, \eta}\to\C$, such that
for all $n\in\mathbb{N}$ we have that $g_{n}$ is
$C_{\eta}$-Lipschitz, $|g_{n}-f_{n}|\leq \eta$, and
$|\widetilde{g}(x+iy)-g(x)|\leq \delta$ for all
$z=x+iy\in\widetilde{U}_{\delta, \eta}$. Since $f_n=0$ for all
$n\geq N$, we may obviously assume that $g_n=0$ for all $n\geq N$,
as well.

Now define a path $\beta:[ 0,+\infty) \rightarrow \ell_{\infty}$
by,
\[
\beta\left(  t\right)  =\left(  1,\cdots,1,\underset{n^{\text{th}%
}\ \text{place}}{\underbrace{t-n+1}},0,0,\cdots
\right)=\sum_{j=1}^{n-1}e_j +(t-n+1)e_n  \text{ if }  n-1\leq
t\leq n.
\]
Clearly the path $\beta$ is a $1$-Lipschitz injection of $[0,
+\infty)$ into $\ell_{\infty}$, with a uniformly continuous (but
not Lipschitz) inverse $\beta^{-1}:\beta([0,+\infty))\to [0,
+\infty)$.

Define a uniformly continuous (not Lipschitz) function $h$ on the
path $\beta$ by $h\left( \beta\left( t\right) \right) =t$ for all
$t\geq 0$, that is $h(y)=\beta^{-1}(y)$ for $y\in\beta([0,
+\infty))$. Then we have $f\left( x\right) =h\left(
\{f_{n}(x)\}_{n=1}^{\infty}\right)$.

\medskip

\subsection{The gluing tube function.}

Now we are going to construct an open tube of radius $2\eta$ (with
respect to the supremum norm $\|\cdot\|_{\infty}$) around the path
$\beta$ in $\ell_{\infty}$, and a real-analytic approximate
extension (with bounded derivative) $H$ of the function $h$
defined on this tube. This construction is meant to be used as
follows: since $|g_n-f_n|\leq \eta$ and
$\{f_{n}(x)\}_{n=1}^{\infty}$ takes values in the path $\beta$,
then $\{g_{n}(x)\}_{n=1}^{\infty}$ will take values in this tube,
and therefore $g(x):=H(\{g_{n}(x)\}_{n=1}^{\infty})$ will
approximate $H(\{f_{n}(x)\}_{n=1}^{\infty})$, which in turn
approximates $h(\{f_{n}(x)\}_{n=1}^{\infty})=f(x)$. Besides, since
$H$ has a bounded derivative on the tube and the functions $g_n$
are $C_{\eta}$-Lipschitz then $g$ will be $C_{\eta}M$-Lipschitz,
where $M$ is an upper bound of $DH$ on the tube.

\medskip

\begin{lemma}\label{tube}
For every $\varepsilon\in (0,1]$ there exist $r, \delta \in (0,
\varepsilon/64)$ and a real-analytic mapping
$G:\ell_{\infty}\to\ell_{\infty}$ with holomorphic extension
$\widetilde{G}:\widetilde{\ell}_{\infty}\to\widetilde{\ell}_{\infty}$
such that:
\begin{enumerate}
\item $G$ diffeomorphically maps the {\em straight tube} $\mathcal{S}$ defined by $\{x\in\widetilde{\ell}_{\infty} :
|x_n|<r \textrm{ for all } n\geq 2, x_1 >-r\} $ onto a {\em
twisted tube} $\mathcal{T}$ around the path $\beta$ in such a way
that
    $$
\{x\in {\ell}_{\infty} : \textrm{dist}(x,
\beta([0,\infty)))<\frac{r}{2}\}\subseteq \mathcal{T}\subseteq
\{x\in {\ell}_{\infty} : \textrm{dist}(x, \beta([0,\infty)))<
2r\}.
    $$
\item We have
    $
\|x-\beta(t)\|_{\infty}\leq r/2 \implies
|e_{1}^{*}G^{-1}(x)-t|\leq\varepsilon.
    $
\item[{(2')}] There is a function $\alpha:\R\to\R$ with a holomorphic extension
$\widetilde{\alpha}:\C\to\C$ such that the function
$H:=\alpha\circ e_{1}^{*}\circ G^{-1}:\mathcal{T}\to\R$ satisfies
$$\|DH(x)\|\leq(1+\varepsilon), \textrm{ and }$$
    $$
\|x-\beta(t)\|_{\infty}<\frac{r}{2} \implies |H(x)-t|\leq
\varepsilon.
    $$
\item The derivatives of the maps $G_{|_{\mathcal{S}}}$ and
$(G_{|_{\mathcal{S}}})^{-1}$ are bounded by $(2+\varepsilon)$ on
the sets $\mathcal{S}$ and $\mathcal{T}$ respectively.
\item $\|\widetilde{G}(x+iy)-G(x)\|_{\infty}\leq \varepsilon$ for all
$x, y\in\mathcal{S}$ with $\|y\|_{\infty}\leq\delta$.
\item A holomorphic extension
$\widetilde{G^{-1}}$ of $G^{-1}$ is defined from
$\widetilde{\mathcal{T}_{\delta}}:=\{u+iv
: u\in\mathcal{T}, v\in\ell_{\infty}, \|v\|_{\infty}<\delta\}$ into
$\widetilde{\ell_{\infty}}$, with the properties that
$$\|\widetilde{G^{-1}}(u+iv)- G^{-1}(u)\|_{\infty}\leq\varepsilon, \textrm{ and }$$
    $$
|\widetilde{H}(u+iv)|=|\widetilde{\alpha}(\widetilde{e_{1}^{*}}(\widetilde{G^{-1}}(u+iv)))|\leq
2\left(H(u)+1\right)
    $$
     for
all $x\in\mathcal{S}, u\in\mathcal{T}$ and $v\in\ell_{\infty}$
such that $u+iv\in\widetilde{\mathcal{T}_{\delta}}$.
\end{enumerate}
\end{lemma}
\noindent {\em Proof.} Let $r\in (0, 1/8)$, and define
$S=[0,2]\times [-r, r]$, $T=[0,1+r]\times [-r,r]\,\cup\,
[1-r,1+r]\times [-r,1]$. Consider the mapping $F:S\to T$ defined
by
    $$
F(x,y)=
  \begin{cases}
    \left(x-yx, y\right) & \text{ if } 0\leq x\leq 1, -r\leq y\leq r, \\
    \left(1-y, x-1+y(2-x)\right) & \text{ if } 1\leq x\leq 2, -r\leq y\leq
    r.
  \end{cases}
    $$
It is clear that $F$ is a Lipschitz homeomorphism from $S$ onto
$T$, with inverse
    $$
F^{-1}(u,v)=
  \begin{cases}
    \left( \frac{u}{1-v}, v\right) & \text{ if } (u,v)\in [0, 1+r]\times [-r,r], u+v\leq 1,  \\
    \left(\frac{2u+v-1}{u}, 1-u\right) & \text{ if } (u,v)\in [1-r, 1+r]\times [-r,1], u+v\geq 1.
  \end{cases}
    $$
Moreover, from these expressions it is immediately checked that
$F$ is $C^{\infty}$ smooth except on the set $\{x=1\}$, that
$F^{-1}$ is also $C^{\infty}$ smooth except on the set
$\{u+v=1\}$, and that $F, F^{-1}\in W^{1, \infty}$. In fact, using
the norm $\|\cdot\|_{\infty}$ in $\mathbb{R}^{2}$, the derivatives
$DF$ and $DF^{-1}$ satisfy that
    $$
\| DF\|\leq 2+r \, \textrm{ a. e. on } S,  \,\,\, \textrm{ and }
\,\,\,
 \| DF^{-1}\|\leq \frac{2}{(1-r)^2} \, \textrm{ a. e. on
} S.
    $$
This implies that $F:S\to T$ is $(2+r)$-Lipschitz (because $S$ is
convex), and that $F^{-1}$ is locally $2/(1-r)^2$-Lipschitz
(because $T$ is locally convex). Given $\varepsilon\in (0, 1/8)$,
by using convolution with mollifiers and observing that
$\lim_{r\to 0}2/(1-r)^2=2=\lim_{r\to 0}(2+r)$, for a given
$\varepsilon\in (0,1)$, we can find $r=\varepsilon/64$ and
$\overline{F}, \overline{G}\in C^{\infty}(\mathbb{R}^{2},
\mathbb{R}^{2})$ satisfying
    $$
\|\overline{F}_{|_S}-F\|_\infty<\frac{\varepsilon}{10} \textrm{ on
}S, \,\, \|\overline{G}_{|_S}-F\|_\infty<\frac{\varepsilon}{10}
\textrm{ on } T,
    $$
and such that $\overline{F}$ is a diffeomorphism from $T$ onto
$\overline{F}(T)$ with inverse
$\overline{G}_{|_{\overline{F}(T)}}$, and $\overline{F}$ and
$\overline{G}$ are locally $(2+\varepsilon/4)$Lipschitz (with
respect to the norm $\|\cdot\|_{\infty}$). Next, by scaling,
translating and appropriately gluing with the isometries $(x, y)$
and $(1-y, x-1)$ (allowing $r$ to be smaller if necessary), one
can obtain a $C^\infty$ mapping $\overline{\Phi}:\R^2\to\R^2$ such
that
\begin{enumerate}
\item[{(i)}] $\overline{\Phi}$ is the identity on the rectangle $[-2r,1-\varepsilon]\times [-2r,
2r]$;
\item[{(ii)}] $\overline{\Phi}$ maps isometrically the rectangle $[1+\varepsilon, 2]\times [-2r,
2r]$ onto the rectangle $[1-2r, 1+2r]\times [\varepsilon,1]$, and
$D\Phi(s,t)(1,0)=(0,1)$ for all $(s,t)\in [1+\varepsilon, 2]\times
[-2r, 2r]$;
\item[{(iii)}] $\overline{\Phi}$ is a diffeomorphism from $[0,2]\times [-2r,
2r]$ onto $\overline{\Phi}\left( [0,2]\times [-2r, 2r]\right)$;
\item[{(iv)}] considering, by abusing notation, that
$\beta:[0,2]\to\R^2\subset\ell_{\infty}$, we have
\begin{eqnarray*}
& & \{(u,v)\in\R^2 : \textrm{dist}((u,v),
\beta([0,2]))<\frac{r}{2}+\frac{r}{64}\}\subseteq\\
& &\overline{\Phi}([0,2]\times [-r, r])\subseteq \{(s,t)\in\R^2
: \textrm{dist}((s,t),
\beta([0,2]))<r-\frac{r}{64}\}, \\
& & \|\overline{\Phi}(t,0)-\beta(t)\|\leq \frac{\varepsilon}{5}
\textrm{ for all } t\in [0,2];
\end{eqnarray*}
\item[{(v)}] $\overline{\Phi}$ is $(2+\varepsilon/2)$-Lipschitz on
the set $[0,2]\times [-r,r]$, and $\overline{\Phi}^{-1}$ is
$(2+\varepsilon/2)$-locally Lipschitz on the set $\{(u,v)\in\R^2
: \textrm{dist}((u,v),
\beta([0,2]))<r\}$
\item[{(vi)}] $\overline{\Phi}(u,v)=(0,0)$ if $|u|, |v|\geq
10$.
\end{enumerate}
Notice that the local isometry property in $(ii)$ is provided by a
$\pi/2$-rotation followed by a translation, which is an affine
isometry with respect to the norm $\|\cdot\|_{\infty}$ as well as
to the Euclidean norm in $\R^2$. In fact this local isometry is
given by $(x,y)\mapsto (1-y, x-1)$.

Observe that such a mapping $\overline{\Phi}$ also satisfies
\begin{enumerate}
\item[{(vii)}] $|e_{1}^{*}\circ\overline{\Phi}^{-1}(x)-t|\leq 3\varepsilon$ whenever
$\|x-\beta(t)\|_{\infty}\leq r, \, t\in [0,2].$
\end{enumerate}

\medskip

Denote $\overline{\Phi}=(\overline{\varphi}, \overline{\psi})$.
Since the coordinate functions $\overline{\varphi},
\overline{\psi}$ have bounded support, their derivatives (of all
orders) are bounded, and in particular $\Phi$ is a Lipschitz
mapping with a Lipschitz derivative.

Now set
    $$
\overline{\Phi}_{(1,2)}(x)=\overline{\varphi}(x_1, x_{2})e_{1}+
\overline{\psi}(x_1, x_{2})e_{2}+\sum_{k\geq 3}x_k e_k,
    $$
and, for $n\geq 1$, define $\overline{\Phi}_{(n+1,
n+2)}:\ell_{\infty}\to\ell_{\infty}$ by
\begin{eqnarray*}
& & \overline{\Phi}_{(n+1, n+2)}(x)=\\
& & \sum_{j=1}^{n}(1-x_{j+1})e_j + \overline{\varphi}(x_1-n,
x_{n+2})e_{n+1} + \overline{\psi}(x_1-n, x_{n+2})e_{n+2}
+\sum_{k\geq n+3}x_k e_k.
\end{eqnarray*}

One can check that the mappings $\overline{\Phi}_{(n, n+1)}$ are
$C^{\infty}$ smooth on $\ell_{\infty}$, satisfy
$\overline{\Phi}_{(n, n+1)}(x)=\overline{\Phi}_{(n+1, n+2)}(x)$
whenever $n +1/4\leq x_1\leq n+3/4$ and $|x_j|<1/6$ for all $j\geq
2$. Moreover, there exists $N>0$ such that
    $$
\|\overline{\Phi}_{(n, n+1)}(x)\|_{\infty} \leq N \textrm{ for all
} x\in\mathcal{S}, n\in\N
    $$
and
    $$
\|D^{k}\overline{\Phi}_{(n, n+1)}(x)\|\leq N \textrm{ for all }
x\in\ell_{\infty}, n\in\N, k\in\{1,2\}.
    $$

Then define $\overline{G}:\mathcal{S}\subset
\ell_{\infty}\to\ell_{\infty}$ by
    $$
\overline{G}(x)=\sum_{n=0}^{\infty}\overline{\theta}(x_1-n)\overline{\Phi}_{(n+1,
n+2)}(x),
    $$
where $\overline{\theta}:\R\to\R$ is a $C^\infty$ function such
that
\begin{enumerate}
\item[{}] $\overline{\theta}(t)=0$ if and only if $t\in (-\infty, -1/2]\cup [3/4, +\infty)$;
\item[{}] $\overline{\theta}(t)=1$ if and only if $t\in [-1/4, 1/2]$;
\item[{}] $\overline{\theta}'(t)>0$ if and only if $t (-1/2, -1/4)$;
\item[{}] $\overline{\theta}(t)=1-\overline{\theta}(t-1)$ if $t\in (1/2, 3/4)$.
\end{enumerate}
Note in particular that the collection of functions
$t\mapsto\overline{\theta}(t-n)$, $n\in\N$, form a $C^\infty$
partition of unity on $\R$. Observe also that, for every $x\in
\ell_{\infty}$ such that $|x_j|\leq 2r$ for $j\geq 2$, there exist
$n_x\in\mathbb{N}$ and a neighborhood $U_x$ of $x$ so that
$\overline{G}(y)=\overline{\Phi}_{n_x, n_x +1}(y)$ for all $y\in
U_x$.

It is not difficult to check that the mapping
$\overline{G}:\mathcal{S}\subset\ell_{\infty}\to\ell_{\infty}$ has
properties $(1), (2)$ and $(3)$ of the statement (with the
slightly sharper bound $\frac{1}{2}+\frac{1}{8}$, instead of
$\frac{3}{4}$, in $(2)$). In particular $\overline{G} ^{-1}$ is
uniformly continuous and there exists $m>0$ such that
    $$
\|D\overline{G}(x)(h)\|\geq m\|h\| \textrm{ for all } x,
h\in\ell_{\infty}.
    $$
But of course $\overline{G}$ is not real analytic. In order to
obtain a required real-analytic function $G$ we shall substitute
$\overline{\varphi}, \overline{\psi}$ and $\overline{\theta}$ with
real-analytic approximations of these functions defined by
appropriate integral convolutions with Gaussian kernels in $\R^2$
and $\R$. Namely, let us define
\begin{enumerate}
\item[{}] $\varphi(s,t)=a_{\kappa}\int_{\R^2}\overline{\varphi}(u,v)\exp\left(-\kappa \left[
(s-u)^{2}+(t-v)^{2}\right]\right)dudv$,
\item[{}] $\psi(s,t)=a_{\kappa}\int_{\R^2}\overline{\psi}(u,v)\exp\left(-\kappa \left[
(s-u)^{2}+(t-v)^{2}\right]\right)dudv$,
\item[{}] $\theta(t)=b_{\kappa}\int_{-\infty}^{+\infty}\overline{\theta}(u)e^{-\kappa
(t-u)^{2}}du$,
\item[{}] ${\Phi}_{(n+1, n+2)}(x)=
\sum_{j=1}^{n}(1-x_{j+1})e_j + {\varphi}(x_1-n, x_{n+2})e_{n+1}
+\\ +{\psi}(x_1-n, x_{n+2})e_{n+2} +\sum_{k\geq n+3}x_k e_k.$
\end{enumerate}
where
\begin{enumerate}
\item[{}] $a_{\kappa}=1/\int_{\R^2}\exp\left(-\kappa \left[
u^{2}+v^{2}\right]\right)dudv, \textrm{ and }$
\item[{}] $b_{\kappa}=1/\int_{-\infty}^{+\infty}e^{-\kappa v^{2}}dv.$
\end{enumerate}
By taking $\kappa$ large enough we may assume that the functions
$\varphi, \psi, \theta$ and their first and second derivatives are
as close to $\overline{\varphi}, \overline{\psi},
\overline{\theta}$ and their first and second derivatives,
respectively, as we want. Therefore the mappings
$\overline{\Phi}_{(n, n+1)}$ and their first and second
derivatives can be taken as close as needed  to $\Phi_{(n, n+1)}$
and their first and second derivatives, say
    $$
\|D^{(k)}\Phi_{(n, n+1)}(x)-D^{(k)}\overline{\Phi}_{(n,
n+1)}(x)\|_{\infty}\leq\varepsilon/2
    $$
for all $x\in\ell_{\infty}$, $k=0,1,2$. Moreover, it is easily
checked that the holomorphic extensions $\widetilde{\varphi},
\widetilde{\psi}$ of $\varphi, \psi$ to $\C^{2}$, as well as the
holomorphic extension of $\theta$ to $\C$ (defined by the same
formulae by letting $s,t\in\C$) satisfy
\begin{eqnarray*}
& &|\varphi(s+ir,t+iw)|\leq e^{\kappa(r^2+w^2)}|\varphi(s,t)|,\\
& &|\psi(s+ir,t+iw)|\leq e^{\kappa(r^2+w^2)}|\psi(s,t)|, \\
& &|\theta(a+ib)|\leq e^{\kappa b^2}|\theta(a)|
\end{eqnarray*}
and in particular we have
\begin{eqnarray*}
& &|\varphi(s+ir,t+iw)|\leq (1+\varepsilon)|\varphi(s,t)|,\\
& &|\psi(s+ir,t+iw)|\leq (1+\varepsilon)|\psi(s,t)|, \\
& &|\theta(a+ib)|\leq (1+\varepsilon)|\theta(a)|
\end{eqnarray*}
provided that $\max\{|r|^{2}+|w|^{2},
|b|^{2}\}<\frac{\log(1+\varepsilon)}{\kappa}$. It follows that the
corresponding holomorphic extension $\widetilde{\Phi}_{(n, n+1)}$
of $\Phi_{(n, n+1)}$ to $\widetilde{\ell}_{\infty}$ (defined by a
similar formula just replacing $\varphi, \psi$ with
$\widetilde{\varphi}, \widetilde{\psi}$) satisfies
    $$
\|\widetilde{\Phi}_{(n, n+1)}(x+iz)\|_{\infty}\leq
(1+\varepsilon)\|\Phi_{(n, n+1)}(x)\|_{\infty}
    $$
for all $x,z\in\ell_{\infty}$ such that $\|z\|_{\infty}\leq
\sqrt{\frac{\log(1+\varepsilon)}{2\kappa}}$.

Now define $$G(x)=\sum_{n=0}^{\infty}\theta(x_1-n)\Phi_{(n+1,
n+2)}(x).$$ Let us see that $G$ is real analytic and has a
holomorphic extension $\widetilde{G}$ to
$\widetilde{\ell}_{\infty}$ defined by
$$\widetilde{G}(x+iz)=\sum_{n=0}^{\infty}\widetilde{\theta}(x_1-n+iz_1)\widetilde{\Phi}_{(n+1, n+2)}(x+iz).$$
It is enough to check that the series of holomorphic functions
defining $\widetilde{G}$ is locally uniformly absolutely
convergent. Taking into account that the mappings
$\widetilde{\Phi}_{(n, n+1)}$ are clearly uniformly bounded on
bounded sets (meaning that for every $R>0$ there exists $K>0$ such
that $\|\widetilde{\Phi}_{(n, n+1)}(x+iz)\|\leq K$ for all $x,
z\in B(0, R)$ and all $n\in\N$), this amounts to showing that
    $$
\sum_{n=0}^{\infty}|\widetilde{\theta}(t-n+is)|<+\infty
    $$
locally uniformly for $t+is\in\C$. Assume $|t|, |s|<R\leq n$,
$u\in [-1/2, 3/4]$, then we have $-\kappa |t-n|^2+2\kappa
(t-n)u\leq -\kappa (n-R)^{2}+\frac{3}{2}\kappa (n+R)$, hence
\begin{eqnarray*}
& &0\leq \theta(t-n)=\frac{1}{\int_{-\infty}^{\infty}e^{-\kappa
v^{2}}dv} \int_{-\infty}^{\infty}\overline{\theta}(u)e^{-\kappa
(t-n-u)^{2}}du\leq\\
& &\frac{1}{\int_{-\infty}^{\infty}e^{-\kappa v^{2}}dv}
e^{-\kappa\left[(n-R)^{2}-\frac{3}{2}
(n+R)\right]}\int_{-1/2}^{3/4}e^{-\kappa u^{2}}du\leq
e^{-\kappa\left[(n-R)^{2}-\frac{3}{2} (n+R)\right]},
\end{eqnarray*} and consequently
    $$
|\widetilde{\theta}(t-n+is)|\leq \theta(t-n)e^{\kappa s^{2}}\leq
e^{-\kappa\left[(n-R)^{2}-\frac{3}{2} (n+R) -R^{2} \right]}.
    $$
Since
    $$
\sum_{n\geq R}^{\infty}e^{-\kappa\left[(n-R)^{2}-\frac{3}{2} (n+R)
-R^{2}\right]}<+\infty,
    $$
it is then clear that the series $
\sum_{n=0}^{\infty}|\widetilde{\theta}(t-n+is)|<
    $
is bounded on $\{t+is: |t|, |s|<R\}$ by an absolutely convergent
numerical series, and therefore it is locally uniformly
convergent.

\medskip

Now let us show that, given $\varepsilon>0$, $G$ satisfies
$\|G(x)-\overline{G}(x)\|_{\infty}\leq\varepsilon$ for all
$x\in\mathcal{S}$, provided $\kappa>0$ is large enough. Let us
first observe that we can take $\kappa>0$ sufficiently large so
that
    $$
\sum_{n=0}^{\infty}|\theta(t-n)-\overline{\theta}(t-n)|\leq\frac{\varepsilon}{2(N+1)}.
    $$
Indeed, on the one hand, if $|t-n|\geq 2$ we have
$\overline{\theta}(t-n)=0$, and
$|t-n|^{2}-\frac{3}{2}|t-n|=|t-n|\left(|t-n|-\frac{3}{2}\right)\geq\frac{1}{2}|t-n|$,
hence
\begin{eqnarray*}
& &|\theta(t-n)-\overline{\theta}(t-n)|= \theta(t-n)=
\frac{1}{\int_{-\infty}^{\infty}e^{-\kappa v^{2}}dv}
\int_{-\infty}^{\infty}\overline{\theta}(u)e^{-\kappa
(t-n-u)^{2}}du\leq\\
& &\frac{1}{\int_{-\infty}^{\infty}e^{-\kappa v^{2}}dv}
e^{-\kappa\left(|n-t|^{2}-\frac{3}{2}
|t-n|\right)}\int_{-1/2}^{3/4}e^{-\kappa u^{2}}du\leq e^{-\kappa
|t-n|/2},
\end{eqnarray*}
consequently
    $$
\sum_{|t-n|\geq 2}|\theta(t-n)-\overline{\theta}(t-n)|\leq
\sum_{|t-n|\geq 2}e^{-\kappa |t-n|/2}\leq 2\sum_{n=2}^{\infty}
e^{-\kappa n/2},
    $$
and because $2\sum_{n=2}^{\infty} e^{-\kappa n/2}\to 0 \, \textrm{
as } \kappa\to +\infty$, we may assume $\kappa$ is large enough so
that
    $$
\sum_{|t-n|\geq
2}|\theta(t-n)-\overline{\theta}(t-n)|\leq\frac{\varepsilon}{4(N+1)}.
    $$
On the other hand, we may also assume $\kappa>0$ is large enough
so that $|\theta(u)-\overline{\theta}(u)|\leq\varepsilon/16(N+1)$
for all $u\in\R$, and clearly there are at most four integers $n$
with $|t-n|<2$, so we have
    $$
\sum_{|t-n|<2}|\theta(t-n)-\overline{\theta}(t-n)|\leq
4\frac{\varepsilon}{16(N+1)}=\frac{\varepsilon}{4(N+1)}.
    $$
By combining the two last inequalities we get
    $$
\sum_{n=0}^{\infty}|\theta(t-n)-\overline{\theta}(t-n)|\leq
\frac{\varepsilon}{2(N+1)}
    $$
as required.

Now, for every $x\in\mathcal{S}$ we can estimate
\begin{eqnarray*}
& &\|G(x)-\overline{G}(x)\|_{\infty}\leq
\|\sum_{n=0}^{\infty}\left(\theta(x_1-n)-\overline{\theta}(x_1-n)\right)
\Phi_{n+1,
n+2}(x)\|+\\
& &+\|\sum_{n=0}^{\infty}\overline{\theta}(x_1-n)\left(
\Phi_{n+1, n+2}(x)-\overline{\Phi}_{n+1, n+2}(x)\right)\| \\
& &\leq \sum_{n=0}^{\infty}\|\Phi_{n+1, n+2}(x)\| \,
|\theta(x_1-n)-\overline{\theta}(x_1-n)| + \\ & & +
\sum_{n=0}^{\infty}\overline{\theta}(x_1-n)\|\Phi_{n+1,
n+2}(x)-\overline{\Phi}_{n+1, n+2}(x)\| \\ & & \leq
(N+1)\sum_{n=0}^{\infty}|\theta(x_1-n)-\overline{\theta}(x_1-n)|+
\frac{\varepsilon}{2}\sum_{n=0}^{\infty}\overline{\theta}(x_1-n)\leq
\frac{\varepsilon}{2}+\frac{\varepsilon}{2}=\varepsilon.
\end{eqnarray*}

By using the facts that the derivatives $\theta', \theta''$ of
$\theta$ decrease exponentially at infinity and approximate
$\overline{\theta}', \overline{\theta}''$, and by performing
similar calculations, one can also show that
    $$
\|DG(x)-D\overline{G}(x)\|\leq \varepsilon, \,\, \textrm{ and }
\,\,
     \|D^2 G(x)-D^2
\overline{G}(x)\|\leq \varepsilon, \textrm{ for all }
x\in\mathcal{S}
    $$
provided $\kappa>0$ is large enough.

\medskip

Let us now see that $G$ is a diffeomorphism from $\mathcal{S}$
onto $\mathcal{T}:=G(\mathcal{S})$. We know that $\overline{G}$ is
of class $C^\infty$ with bounded derivatives of all orders, in
particular $\overline{G}$ is Lipschitz and has a Lipschitz
derivative. Moreover $D\overline{G}(x)$ is a linear isomorphism
for all $x\in \mathcal{S}$ and the mapping $x\mapsto
\|D\overline{G}(x)^{-1}\|$ is bounded above on $\mathcal{S}$ by a
number $M>0$. This implies that if
$L:\ell_{\infty}\to\ell_{\infty}$ is a linear mapping such that
$\|L\|<1/M$ then $D\overline{G}(x)+L$ is a linear isomorphism as
well, for every $x\in \mathcal{S}$. Since we may assume that
$\|D\overline{G}(x)-DG(x)\|<1/M$ for all $x\in\mathcal{S}$, we
have that $DG(x)$ is a linear isomorphism for all
$x\in\mathcal{S}$

Now we recall that a standard proof of the inverse mapping theorem
(the one that uses the fixed point theorem for contractive
mappings) shows that if a function $F$ (defined on an open set $U$
of a Banach space $X$ and taking values in $X$) has the property
that $DF(x)$ is a linear isomorphism for all $x$, the mappings
$x\mapsto DF(x)$ and $x\mapsto DF(x)^{-1}$ are bounded, and the
mapping $x\mapsto DF(x)$ is Lipschitz, then there exist uniform
lower bounds $r, s>0$ (depending only on $\textrm{Lip}(DF)$ and
the bounds for $DF, (DF)^{-1}$) such that $F$ maps
diffeomorphically each open ball $B(x,r)$ contained in $U$ onto an
open subset of $X$ containing the ball $B(F(x), s)$; see for
instance \cite[Proposition 2.5.6]{AMR}. Since $D\overline{G},
D^{2}\overline{G}$ are bounded, and $DG, D^2 G$ are being assumed
to be close enough to these functions for large $\kappa$, it
follows that $DG, D^2 G$ are uniformly bounded as well for all
$\kappa>0$ large enough, and we can apply the mentioned fact with
$F=G$ to conclude that there exists $r>0$ such that, for all
$\kappa>0$ large enough, $G$ is a diffeomorphism from each open
ball $B(x,r)\subset\mathcal{S}$ onto an open subset of
$\ell_{\infty}$, and in particular $G$ is uniformly locally
injective (at distances less that $r$). The same property is of
course true of $\overline{G}$.

Thus, in order to prove that $G$ maps $\mathcal{S}$
diffeomorphically onto its image it only remains to be seen that
$G$ is globally injective as well (at distances greater than $r$,
at least for $\kappa>0$ large enough). Clearly we have $0<\inf\{
\|\overline{G}(v)-\overline{G}(w)\|_{\infty} : v, w\in\mathcal{S},
\|v-w\|_{\infty}\geq r\}$, so we can take $\varepsilon$ with
$$0<\varepsilon<\frac{1}{4}\inf\{ \|\overline{G}(v)-\overline{G}(w)\|_{\infty} : v, w\in\mathcal{S},
\|v-w\|_{\infty}\geq r\},$$ and of course we may assume $\kappa>0$
is large enough so that
$\|G(x)-\overline{G}(x)\|_{\infty}<\varepsilon$ for all
$x\in\mathcal{S}$. Now take $x, y\in \mathcal{S}$ with $x\neq y$.
If $G(x)=G(y)$ then necessarily $\|x-y\|_{\infty}\geq r$, and we
have
\begin{eqnarray*}
& & \|\overline{G}(x)-\overline{G}(y)\|_{\infty}=\\
& &\|\overline{G}(x)-G(x)+G(y)-\overline{G}(y)\|_{\infty}\leq
\|\overline{G}(x)-G(x)\|_{\infty}+\|G(y)-\overline{G}(y)\|_{\infty}\\
& & \leq 2\varepsilon
< \inf\{ \|\overline{G}(v)-\overline{G}(w)\|_{\infty} : v, w\in\mathcal{S},
\|v-w\|\geq r\},
\end{eqnarray*}
a contradiction.

This proves the first part of $(1)$. The second part of $(1)$
follows easily from the definitions of $\overline{\Phi}$,
$\overline{G}$ and the fact that $G$ approximates $\overline{G}$
(we may assume $0<\varepsilon<1/64$ and $\kappa>0$ large enough so
that $\|G(x)-\overline{G}(x)\|\leq\varepsilon$ for all
$x\in\mathcal{S}$).

Now define, for $t\in\R$, $z\in\C$, the functions
    $$
\overline{\alpha}(t)=
  \begin{cases}
    0 & \text{ if } t\leq 0\\
    n & \text{ if } t\in [n, n+2\varepsilon], \\
    n +\frac{1}{1-4\varepsilon}(t-n-2\varepsilon) & \text{ if } t\in [n+2\varepsilon, n+1-2\varepsilon], \\
    n+1 & \text{ if } t\in [n+1-2\varepsilon, n+1], \, n=0,1, 2,
    ...,
  \end{cases}
    $$
$$
\widetilde{\alpha}(z)=\frac{\int_{\mathbb{R}}\overline{\alpha}(s)e^{-\kappa(z-s)^2}ds}{\int_{\mathbb{R}}e^{-\kappa
s^2}ds},
$$
and put $\alpha=\widetilde{\alpha}_{|_{\R}}$. Observe that the
functions $\overline{\alpha}$ and $\alpha$ are
$\frac{1}{1-4\varepsilon}$-Lipschitz. Also, we have
$|\overline{\alpha}(t)-t|\leq 2\varepsilon$, and hence
$|\alpha(t)- t|\leq 3\varepsilon$  for all $t\geq 0$, if $\kappa$
is large enough. Besides $\overline{\alpha}'(t)=0$ whenever
$|t-n|< 2\varepsilon$, hence we can also assume $\kappa$ is large
enough so that $|\alpha'(t)|\leq \varepsilon/4$ whenever
$|t-n|\leq3\varepsilon/2$.

Set $\overline{H}:=\alpha\circ e_{1}^{*}\circ \overline{G}^{-1}$,
and $H=\alpha\circ e_{1}^{*}\circ {G}^{-1}$. From property $(vii)$
above we get
    $$
|e_{1}^{*}\circ \overline{G}^{-1}(x)-t|\leq 4\varepsilon \textrm{
whenever } \|x-\beta(t)\|\leq r/2,
    $$
which in particular shows $(2)$. Hence we have, for
$\|x-\beta(t)\|_{\infty}\leq r/2$,
    $$
|\overline{H}(x)-t|\leq |\alpha(e_1^{*}(\overline{G}^{-1}(x)))-
e_1^{*}(G^{-1}(x))|+|e_1^{*}(\overline{G}^{-1}(x))-t|\leq
3\varepsilon+4\varepsilon=7\varepsilon,
    $$
By using the fact that $G^{-1}$ approximates $\overline{G}^{-1}$
for $\kappa>0$ large enough\footnote{This is a consequence of the
following exercise: if $f_k$ is a sequence of diffeomorphisms
which uniformly converges to a diffeomorphism $f$ with a uniformly
continuous inverse $f^{-1}$, then $f_{k}^{-1}$ uniformly converges
to $f^{-1}$.}, we also deduce $$ |H(x)-t|=
|\alpha(e_1^{*}(G^{-1}(x)))-t|\leq 8\varepsilon, \textrm{ whenever
} \|x-\beta(t)\|_{\infty}\leq r/2.
    $$
Up to a change of $\varepsilon$ and $r$, this shows the second
part of property $(2')$.

Next, the bounds in $(3)$ can be obtained from similar bounds for
$D\overline{G}, (D\overline{G})^{-1}$, and the facts that these
derivatives are uniformly approximated by $DG$, $(DG)^{-1}$,
respectively, for $\kappa>0$ large enough.\footnote{This is a
consequence of the following (easy to prove) fact: if $f_k$ is a
sequence of diffeomorphisms which uniformly converges to a
diffeomorphism $f$ such that there exists $m>0$ with
$\|Df(x)(h)\|\geq m\|h\|$ for all $x, h$, then the sequence of
derivatives $D(f_{k}^{-1})$ converges to $D(f^{-1})$.} Then we can
also show the first part of $(2')$: if $|e_{1}^{*}G^{-1}(x)-n|\leq
3\varepsilon/2$ for some $n\in\N$ then
    $$
\|DH(x)\|\leq|\alpha'(e_{1}^{*}G^{-1}(x))|\|DG^{-1}(x)\|\leq\frac{\varepsilon}{4}(2+\varepsilon)\leq\varepsilon,
    $$
while, if $|e_{1}^{*}G^{-1}(x)-n|\geq 3\varepsilon/2$ for all
$n\in\N$ then we have that $\overline{G}^{-1}$ is an affine
isometry on a neighborhood of $x$, hence
    $$
\|DG^{-1}(x)\|\leq
\|D\overline{G}^{-1}(x)\|+\varepsilon=1+\varepsilon,
    $$
which implies that $\|DH(x)\|\leq
\frac{1+\varepsilon}{1-4\varepsilon}$. Since $\lim_{\varepsilon\to
0}\frac{1+\varepsilon}{1-4\varepsilon}=1$, the first part of
$(2')$ follows up to a change of $\varepsilon$ and $r$.

\medskip

As for property $(4)$, we shall first check that $D\widetilde{G}$
is bounded on a neighborhood of $\mathcal{S}$ in
$\widetilde{\ell}_{\infty}$ of the form $\{x+iz: x\in \mathcal{S},
z\in\ell_{\infty}, \|z\|\leq 1\}$, so $\widetilde{G}$ is Lipschitz
on this set. Indeed, on the one hand, if $t,s\in\R, |s|\leq 1$, we
have
\begin{eqnarray*}
&
&|(\widetilde{\theta})'(t+is-n)|=\\
& &\left|\frac{1}{\int_{-\infty}^{\infty}e^{-\kappa v^{2}}dv}
\int_{-\infty}^{\infty}\overline{\theta}(u)2(t+is-n-u)e^{-\kappa
(t+is-n-u)^{2}}du\right| \leq\\
& &\frac{1}{\int_{-\infty}^{\infty}e^{-\kappa v^{2}}dv}
2e^{\kappa}\left(|t-n|+\frac{3}{4}+1\right)e^{-\kappa\left(|t-n|^{2}-\frac{3}{2}
|t-n|\right)}\int_{-1/2}^{3/4}e^{-\kappa u^{2}}du\leq\\
&
&2e^{\kappa}\left(|t-n|+2\right)e^{-\kappa\left(|t-n|^{2}-\frac{3}{2}
|t-n|\right)}.
\end{eqnarray*}
For $|t-n|<2$ this expression is bounded by $8e^{4\kappa}$, and
since there are at most four integers $n$ with $|t-n|<2$ we have
    $$
\sum_{|t-n|<2}|(\widetilde{\theta})'(t+is-n)|\leq 32e^{4\kappa}.
    $$
And for $|t-n|\geq 2$ we can estimate
\begin{eqnarray*}
& & \sum_{|t-n|\geq 2}|(\widetilde{\theta})'(t+is-n)|\leq
\sum_{|t-n|\geq
2}2e^{\kappa}\left(|t-n|+2\right)e^{-\kappa\left(|t-n|^{2}-\frac{3}{2}
|t-n|\right)}\leq\\
& & \sum_{|t-n|\geq 2}2e^{\kappa}\left(|t-n|+2\right)e^{-\kappa
|t-n|/2}\leq 2e^{\kappa} \sum_{m=1}^{\infty}(m+2)e^{-\kappa
m/2}<+\infty.
\end{eqnarray*}
Therefore there exists $0<C:=32e^{4\kappa}+2e^{\kappa}
\sum_{m=1}^{\infty}(m+2)e^{-\kappa m/2}<+\infty$ such that
    $$
\sum_{n=0}^{\infty}|(\widetilde{\theta})'(t+is-n)|\leq C
    $$
for all $t,s\in\R$ with $|s|\leq 1$. We also have
    $$
\sum_{n=0}^{\infty}|\widetilde{\theta}(t+is-n)|\leq C' \,\,
\textrm{ for all } \, t,s\in\R, \, |s|\leq 1
    $$
for some $C'>0$.

On the other hand, it is easy to see that there exists $C''>0$
such that
    $$
\|D\widetilde{\Phi}(x+iz)\|\leq C'' \,\, \textrm{ and } \, \,
\|\widetilde{\Phi}(x+iz)\|\leq C''
    $$
for all $x\in\mathcal{S}$, $z\in\ell_{\infty}$, $\|z\|\leq 1$.
Therefore we get
\begin{eqnarray*}
& & \|D\widetilde{G}(x+iz)\|\leq
\sum_{n=0}^{\infty}|(\widetilde{\theta})'(x_1+iz_1-n)| \,
\|\widetilde{\Phi}_{(n+1, n+2)}(x+iz)\|+\\
& &+\sum_{n=0}^{\infty}|\widetilde{\theta}(x_1+iz_1-n)| \,
\|D\widetilde{\Phi}_{(n+1, n+2)}(x+iz)\|\leq\\
& & \sum_{n=0}^{\infty}|(\widetilde{\theta})'(x_1+iz_1-n)| C''+
\sum_{n=0}^{\infty}|\widetilde{\theta}(x_1+iz_1-n)|C''\leq
C''(C+C')
\end{eqnarray*}
for all $x\in\mathcal{S}, z\in\ell_{\infty}, \|z\|\leq 1$.

Since $D\widetilde{G}$ is bounded on the convex set $\{x+iz: x\in
\mathcal{S}, z\in\ell_{\infty}, \|z\|\leq 1\}$, it immediately
follows from the mean vale theorem that for every $\varepsilon>0$
there exists $\delta>0$ such that if $x\in\mathcal{S},
z\in\ell_{\infty}$ and $\|z\|\leq\delta$ then
$\|\widetilde{G}(x+iz)-G(x)\|\leq\varepsilon$. This shows $(4)$.

\medskip

Similar calculations show that $D^{2}\widetilde{G}$ is also
bounded on $\{x+iz: x\in \mathcal{S}, z\in\ell_{\infty}, \|z\|\leq
1\}$. And we already know that $DG$ and $(DG)^{-1}$ are bounded on
$\mathcal{S}$. Then, according to Proposition 2.5.6 of \cite{AMR}
(or rather, by an analogous statement for holomorphic mappings,
which can be proved in the same fashion by refining a standard
proof of the inverse mapping theorem for holomorphic mappings),
there exist $r', s'>0$ (depending only on the bounds for $DG$ and
$(DG)^{-1}$ on $\mathcal{S}$ and on the bound of
$D^{2}\widetilde{G}$ on $\{x+iz: x\in \mathcal{S},
z\in\ell_{\infty}, \|z\|\leq 1\}$) such that, for every
$x\in\mathcal{S}$, the mapping $\widetilde{G}$ is a holomorphic
diffeomorphism from the ball $B_{\widetilde{\ell_{\infty}}}(x,
r')$ onto an open subset of $\widetilde{\ell_{\infty}}$ which
contains the ball $B_{\widetilde{\ell_{\infty}}}(G(x), s')$. In
particular, for every $y=G(x)\in\mathcal{T}$ there exists a
holomorphic extension $(\widetilde{G^{-1}})_{y}$ of $G^{-1}$
defined on the ball $B_{\widetilde{\ell_{\infty}}}(G(x), s')$ and
which maps this ball diffeomorphically within the ball
$B_{\widetilde{\ell_{\infty}}}(x, r')$.

Now define
$\widetilde{G^{-1}}(u+iv)=(\widetilde{G^{-1}})_{y}(u+iv)$ if
$u+iv\in B_{\widetilde{\ell_{\infty}}}(y, s')$ for some
$y\in\mathcal{T}$. This mapping is well defined. {\bf For}, if
$w=u+iv\in B_{\widetilde{\ell_{\infty}}}(y_1, s')\cap
B_{\widetilde{\ell_{\infty}}}(y_2, s')$ with $y_1,
y_2\in\mathcal{T}, y_1\neq y_2$, then there exist
$y_3=G(x_3)\in\mathcal{T}$ and $t\in (0, s')$ such that
$B_{\widetilde{\ell_{\infty}}}(y_3, t)\subset
B_{\widetilde{\ell_{\infty}}}(y_1, s')\cap
B_{\widetilde{\ell_{\infty}}}(y_2, s')$, and since $\widetilde{G}$
maps diffeomorphically the ball
$B_{\widetilde{\ell_{\infty}}}(x_3, r')$ onto an open set
containing $B_{\widetilde{\ell_{\infty}}}(y_3, t)$, and the
mappings $(\widetilde{G^{-1}})_{y_1}, (\widetilde{G^{-1}})_{y_2},
(\widetilde{G^{-1}})_{y_3}$ are local inverses of $\widetilde{G}$
defined on $B_{\widetilde{\ell_{\infty}}}(y_3, t)$ and taking
values in $B_{\widetilde{\ell_{\infty}}}(x_3, r')$, we necessarily
have $(\widetilde{G^{-1}})_{y_1}=(\widetilde{G^{-1}})_{y_2}=
(\widetilde{G^{-1}})_{y_3}$ on the open ball
$B_{\widetilde{\ell_{\infty}}}(y_3, t)$, by uniqueness of the
inverse. But then, by the identity theorem for holomorphic
mappings, we must have
$(\widetilde{G^{-1}})_{y_1}=(\widetilde{G^{-1}})_{y_2}$ on the
open connected set $B_{\widetilde{\ell_{\infty}}}(y_1, s')\cap
B_{\widetilde{\ell_{\infty}}}(y_2, s')$, and in particular
$(\widetilde{G^{-1}})_{y_1}(w)=(\widetilde{G^{-1}})_{y_2}(w)$.
Therefore $\widetilde{G^{-1}}$ is a holomorphic extension of
$G^{-1}$ to the open neighborhood
$\widetilde{\mathcal{T}_{s'}}:=\{u+iv
: y\in\mathcal{T}, v\in\ell_{\infty}, \|v\|< s'\}$ of
$\mathcal{T}$ which maps $\widetilde{\mathcal{T}_{s}}$ into
$\widetilde{\mathcal{S}_r'}:=\{x+iz : x\in\mathcal{S},
z\in\ell_{\infty}, \|z\|<r'\}$. Hence
    $$
\|\widetilde{G^{-1}}(u+iv)-G^{-1}(u)\|<r', \eqno(*)
    $$
for all $u\in\mathcal{T}$, $v\in\ell_{\infty}$ with $\|v\|<s'$.
Obviously we can assume $r'<\varepsilon$, so the first part of
$(5)$ is proved. Finally, in order to check the second part of
$(5)$, observe that, for all $t, \nu\in\R$, we have
    $$
|\widetilde{\alpha}(t+i\nu)|\leq e^{\kappa \nu^{2}}\alpha(t).
    $$
Therefore, for all $t+s+i\nu\in\C$, with $|s|\leq
1/\textrm{Lip}(\alpha)$ we have
    $$
|\widetilde{\alpha}(t+s+i\nu)|\leq e^{\kappa \nu^{2}}\alpha(t+s)
\leq e^{\kappa \nu^{2}}\left(
\alpha(t)+\textrm{Lip}(\alpha)|s|\right)\leq e^{\kappa
\nu^{2}}\left(\alpha(t)+1\right).
    $$
We can also assume $0<r'\leq\min\{\frac{1}{\textrm{Lip}(\alpha)},
\sqrt{\frac{\log 2}{\kappa}}\}$, so we get
    $$
|\widetilde{\alpha}(t+z)|\leq 2(\alpha(t)+1) \, \textrm{ if } \,
t\in\R, \, z\in\C, \, |z|\leq r'.
    $$
By combining this estimation with $(*)$ we obtain
$|\widetilde{\alpha}(\widetilde{e_{1}^{*}}(\widetilde{G^{-1}}(u+iv)))|\leq
2\left(\alpha(e_{1}^{*}(G^{-1}(u)))+1\right)$ for all
$u\in\mathcal{T}$, $v\in\ell_{\infty}$ with $\|v\|<s'$, which
establishes the second part of $(5)$. \hspace{8.4cm} $\boxtimes$

\bigskip

\subsection{Proof of the main result in the case of a
$1$-Lipschitz, bounded function.}

Let us continue with the proof of our main result in the case of a
nonnegative, bounded, $1$-Lipschitz function. Fixing
$\varepsilon=1$, there exist $\delta_1, r_1>0$ so that properties
$(4)$ and $(5)$ of the preceding Lemma hold. Now, given a bounded
$1$-Lipschitz function $f:X\to\R$, we may assume (up to the
addition of a constant) that $f$ takes values in the interval $[0,
+\infty)$. Let $N\in\N$ be such that $N\geq f(x)\geq 0$ for all
$x\in X$. By Proposition \ref{intermediate result}, applied with
$\eta=r_1/2$ and this $\delta_1$, we can find real analytic
functions $g_1, ..., g_N$, with holomorphic extensions
$\widetilde{g}_1, ... \widetilde{g}_N$ defined on
$\widetilde{U}:=\widetilde{U}_{\delta_1, r_{1}/2}\supset X$, such
that
\begin{enumerate}
\item $|f_i(x)-g_i(x)|\leq r_{1}/2$ for all $x\in X$.
\item $g_i$ is $C_{r_1/2}$-Lipschitz.
\item $|\widetilde{g_i}(x+iy)-g_i(x)|\leq \delta_1$ for all $z=x+iy\in\widetilde{U}$.
\end{enumerate}

Now let us define $g$ by $g=\alpha\circ e_{1}^{*}\circ G^{-1}
\circ\{g_i\}_{i=1}^{\infty}=H\circ\{g_i\}_{i=1}^{\infty}$, where
we understand $g_i=0$ for all $i>N$.

Note that, in spite of $G^{-1}$ having a bounded derivative,
because $\mathcal{T}$ is not convex we cannot deduce that $G^{-1}$
is Lipschitz. As a matter of fact this mapping is not Lipschitz,
though it is uniformly continuous. Nevertheless, since $X$ is
convex and the derivative of the function $H$ is bounded by
$1+\varepsilon=2$, we do have that $g$ is Lipschitz. And since the
mapping
    $$
X\ni x\mapsto \left( g_1(x), ..., g_N(x), 0, 0,
...\right)\in\ell_{\infty}
    $$
is obviously $C_{r_1/2}$-Lipschitz, it follows that
$g=H\circ\{g_i\}_{i=1}^{\infty}$ is $2C_{r_1/2}$-Lipschitz.

Because $\left(  f_{1}\left(  x\right)  ,\cdots, f_{N}\left(
x\right), 0, 0, \cdots \right)  \in\beta\left(  \left[  0,N\right]
\right) ,$ and also $\left\vert f_{i}-g_{i}\right\vert \leq
r_{1}/2,$ we have $\left( g_{1}\left(  x\right) ,\cdots,
g_{N}\left( x\right), 0, 0, \cdots \right)  \in \mathcal{T}$ and,
by property $(2)$ of the preceding Lemma,
    $$
|g(x)-f(x)|=|H(\{g_n(x)\}_{n=1}^{\infty})-h(\{f_n(x)\}_{n=1}^{\infty})|\leq
\varepsilon=1.
    $$
The function $g$ is clearly real analytic, with holomorphic
extension $\widetilde{g}(z)=(\widetilde{\alpha}\circ
\widetilde{e}_{1}^{*}\circ
\widetilde{G^{-1}})(\widetilde{g}_{1}(z), \widetilde{g}_{2}(z),
..., \widetilde{g}_{N}(z), 0, 0, ...)$ defined on $\widetilde{U}$.
And, because $|\widetilde{g_i}(x+iy)-g_i(x)|\leq \delta$ for all
$z=x+iy\in\widetilde{U}$, we have, using property $(5)$ of the
preceding Lemma, that $|\widetilde{g}(x+iy)|\leq 2(|g(x)|+1)$ for
all $x+iy\in\widetilde{U}$.

By resetting $C_{r_1/2}$ to $C:=2C_{r_1/2}$, we have thus proved
the following version of our main result for bounded functions.

\begin{proposition}
Let $X$ be a Banach space having a separating polynomial. Then
there exist $C\geq 1$ and an open neighborhood $\widetilde{U}$ of
$X$ in $\widetilde{X}$ such that, for every $1$-Lipschitz, bounded
function $f:X\to\R$, there exists a real analytic function
$g:X\to\mathbb{R}$, with holomorphic extension $\widetilde
{g}:\widetilde{U}\to\mathbb{C}$, such that
\begin{enumerate}
\item $|f(x)-g(x)|\leq 1$ for all $x\in X$.
\item $g$ is $C$-Lipschitz.
\item $|\widetilde{g}(x+iy)|\leq 2(|g(x)|+ 1)$ for all $z=x+iy\in\widetilde{U}$.
\end{enumerate}
\end{proposition}

\bigskip

\subsection{Proof of the main result in the case of an unbounded Lipschitz function.}

Let $\varepsilon\in (0,1)$ be fixed, and let $f:X\to\R$ be an
unbounded $1$-Lipschitz function. For $n\in\N, n\geq 2$, let us
define the {\em crowns}
    $$
C_n=C_n^{\varepsilon} := \{x\in X \, : \,
\frac{2^{n-1}}{\varepsilon}\leq Q(x) \leq
\frac{2^{n+1}}{\varepsilon}\},
    $$
and for $n=1$ set
    $
C_1=C_1^{\varepsilon}=\{ x\in X \, : \, Q(x)\leq 4/\varepsilon\}.
    $
Let $f_n$ denote a bounded, $1$-Lipschitz extension of
$f_{|_{C_n}}$ to $X$ (defined for instance by $x\mapsto\max\{
-\|f_{|_{C_n}}\|_{\infty}, \min\{ \|f_{|_{C_n}}\|_{\infty}, \,
\inf_{y\in C_n}\{ f(y)+\|x-y\|\}\, \}\, \}$). According to the
preceding Proposition there exist an open neighborhood
$\widetilde{U}=\widetilde{U}_{\varepsilon}$ of $X$ in
$\widetilde{X}$, and real analytic functions $g_n :X\to \R$, with
holomorphic extensions $\widetilde{g_n}: \widetilde{U}\to \C$,
such that
\begin{enumerate}
\item $|f_n(x)-g_n(x)|\leq 1$ for all $x\in X$.
\item $g_n$ is $C$-Lipschitz.
\item $|\widetilde{g_n}(x+iy)|\leq 2(|g_n(x)|+ 1)$ for all $z=x+iy\in\widetilde{U}$.
\end{enumerate}

\medskip

For $n=1$, let $\overline{\theta}_{1}:\R\to [0,1]$ be a $C^\infty$
function such that
\begin{enumerate}
\item $\overline{\theta}_{1}(t)=1 \, \iff \, t\in [0, 2/\varepsilon]$;
\item $\overline{\theta}_{1}(t)>0 \iff t\in (-1/\varepsilon, 4/\varepsilon)$;
\item $\overline{\theta}_{1}'(t)<0 \iff t\in (2/\varepsilon, 4/\varepsilon)$;
\item $\textrm{Lip}(\overline{\theta}_{1})\leq 1$.
\end{enumerate}
Now let $\overline{\theta}_n:\R\to [0,1]$, $n\in\N$, $n\geq 2$, be
a sequence of $C^\infty$ functions with the following properties:
\begin{enumerate}
\item $\overline{\theta}_{n}(t)>0 \, \iff \, t\in (2^{n-1}/\varepsilon,
2^{n+1}/\varepsilon)$;
\item $\overline{\theta}_{n}'(t)>0 \iff t\in (2^{n-1}/\varepsilon, 2^{n}/\varepsilon)$;
\item $\overline{\theta}_{n}'(t)<0 \iff t\in (2^n/\varepsilon, 2^{n+1}/\varepsilon)$;
\item  $\overline{\theta}_{n}(2^n/\varepsilon)=1$;
\item $\overline{\theta}_{n}(t)=1-\overline{\theta}_{n-1}(t)$ whenever $t\in (2^{n-1}/\varepsilon,
2^{n}/\varepsilon)$;
\item $\textrm{Lip}(\overline{\theta}_{n})\leq \varepsilon/2^{n-2}$.
\end{enumerate}
The functions $\overline{\theta}_n$ form a partition of unity on
$\R$, and
    $$
\sum_{n=1}^{\infty}\textrm{Lip}(\overline{\theta}_n)\leq
3\varepsilon.
    $$
Then the functions $x\mapsto \overline{\theta}_{n}(Q(x))$,
$n\in\N$ also form a partition of unity subordinated to the
covering $\bigcup_{n\in \N}C_n =X$, and the sum of the Lipschitz
constants of these functions is bounded by
$3\varepsilon\textrm{Lip}(Q)$.

Define real analytic functions $\theta_n :\R\to [0,1]$ by
    $$
\theta(t)=a_n \int_{\R}\overline{\theta}_{n}(s)e^{-\kappa_n
(t-s)^{2}}ds,
    $$
where
    $
a_n := \int_{\R}e^{-\kappa_n s^{2}}ds,
    $
and $\kappa_n$ is large enough so that
\begin{enumerate}
\item $e^{-\kappa_n} \leq \varepsilon/2^{n+3}(1+\|g_n\|_{\infty})$;
\item $|\theta_{n}(t)- \overline{\theta}_{n}(t)|\leq \varepsilon/2^{n+3}(1+\|g_n\|_{\infty})$ for all $t\in\R$;
\item $|\theta_{n}'(t)-\overline{\theta}_{n}'(t)|\leq
\varepsilon/2^{n+3}\textrm{Lip}(Q)(1+\|g_n\|_{\infty})$ for all
$t\in\R$.
\end{enumerate}
Also,
    $
\textrm{Lip}(\theta_n)=\textrm{Lip}(\overline{\theta}_n),
    $
and hence $\sum_{n=1}^{\infty}\textrm{Lip}(\theta_n)\leq
3\varepsilon$. Let us define $g:X\to\R$ and
$\widetilde{g}:\widetilde{U}\to\C$ by
    $$
g(x)=\sum_{n=1}^{\infty}\theta_n (Q(x)) g_n (x), \,\, \textrm{ and
} \,\, \widetilde{g}(z)=\sum_{n=1}^{\infty}\widetilde{\theta}_n
(\widetilde{Q}(z)) \widetilde{g}_n (z),
    $$
where
    $
\widetilde{\theta}(u+iv)= a_n
\int_{\R}\overline{\theta}_{n}(s)e^{-\kappa_n (u+iv-s)^{2}}ds.
    $

In order to see that $g$ is well defined and real-analytic, with
holomorphic extension $\widetilde{g}$ defined on $\widetilde{U}$,
it is enough to show that the series of holomorphic mappings
defining $\widetilde{g}$ is locally uniformly and absolutely
convergent on $\widetilde{U}$. And, since
$|\widetilde{g_n}(z)|\leq 2(\|g_n\|_{\infty} +1)$ for all
$z\in\widetilde{U}$, it is sufficient to check that
    $$
\sum_{n=1}^{\infty}|\widetilde{\theta}_{n}(\widetilde{Q}(z))|\,
(1+\|g_{n}\|_{\infty})<+\infty
    $$
locally uniformly on $\widetilde{U}$.

We may assume that $\textrm{Lip}(\widetilde{Q})\leq C$. Then, for
each $x\in X$, according to Lemma \ref{separating function Q
lemma}, we can write
    $$
\widetilde{Q}(x+z)=Q(x)+Z_x, \textrm{ where } Z_x \in\C, \,
|Z_x|\leq C\|z\|_{\widetilde{X}}.
    $$
Fix $x\in X$. There exists $n_x \in\N$ so that $x\in C_{n_x}$ and
in particular $Q(x)\leq 2^{n_x +1}/\varepsilon$.

Now, if $n\geq n_x +3$ and $s\in
\textrm{support}(\overline{\theta}_n)=[2^{n-1}/\varepsilon,
2^{n+1}/\varepsilon]$ we have $s\geq 2^{n_x +2}/\varepsilon$ and
therefore, for $\|z\|_{\widetilde{X}}\leq 1/2C$,
\begin{align*}
&  \operatorname{Re}\left(  \widetilde{Q}\left(  x+z\right)
-s\right)  ^{2}\\
& \\
&  =\left(  Q\left(  x \right)  -s \right)  ^{2}+2\left( Q\left(
x \right)  -s \right)  \operatorname{Re}Z_{x}+\operatorname{Re}%
(Z_{x}^{2})\\
& \\
&=\left(Q(x)-s+ \operatorname{Re}Z_{x} \right)^{2}- \left(
\operatorname{Re}Z_{x} \right)^{2}+\operatorname{Re}(Z_{x}^{2})\\
&\geq \left(Q(x)-s+ \operatorname{Re}Z_{x} \right)^{2}-\frac{1}{2}
\\
& \geq \frac{1}{2}(s-Q(x)-
\operatorname{Re}Z_{x})^{2}+\frac{1}{2}(\frac{2^{n_x
+2}}{\varepsilon}- \frac{2^{n_x +1}}{\varepsilon}
-\frac{1}{2})^{2}-\frac{1}{2}\\
& \geq \frac{1}{2}(s-Q(x)- \operatorname{Re}Z_{x})^{2} + 1.
\end{align*}
Hence, for  $n\geq n_x +3$ and $\|z\|_{\widetilde{X}}\leq 1/2C$,
we have, according to the choice of $\kappa_n$, that
\begin{eqnarray*}
& &|\widetilde{\theta}_{n}(\widetilde{Q}(x+z))| \leq a_n
\int_{\R}\overline{\theta}_{n}(s)e^{-\kappa_n
\operatorname{Re}(\widetilde{Q}(x+z)-s)^{2}}ds\\
& &\leq a_n \int_{\R} e^{-\kappa_n \left( \frac{1}{2}(s-Q(x)-
\operatorname{Re}Z_{x})^{2} + 1 \right)}ds\\
& & = e^{-\kappa_n} a_{n} \int_{\R} e^{-\kappa_n
\frac{1}{2}(s-Q(x)- \operatorname{Re}Z_{x})^{2}}ds= e^{-\kappa_n}
a_{n} \int_{\R}
e^{-\kappa_n u^{2}} \sqrt{2}du\\
& &=\sqrt{2}e^{-\kappa_n}\leq
\sqrt{2}\frac{1}{2^{n}(1+\|g_n\|_{\infty})}.
\end{eqnarray*}
Then it is clear that the series
    $$
\sum_{n=1}^{\infty}|\widetilde{\theta}_{n}(\widetilde{Q}(x+z))|\,
(1+\|g_{n}\|_{\infty})
    $$
is uniformly convergent for $\|z\|_{\widetilde{X}}\leq 1/2C$.
Since we can obviously assume $\widetilde{U}\subset \{ x+iy : x\in
X, \|y\|\leq 1/2C\}$ (and in fact this is always the case, see the
proof of Lemma \ref{construction of the sup partition of unity}),
this argument shows that the series
    $$
\sum_{n=1}^{\infty}\widetilde{\theta}_n (\widetilde{Q}(z))
\widetilde{g}_n (z)
    $$
defines a holomorphic function $\widetilde{g}$ on $\widetilde{U}$.

\medskip

Next let us show that the real analytic function $g$ approximates
$f$ on $X$. Let us define an auxiliary $C^\infty$ function
$\overline{g}$ by
    $$
\overline{g}(x)=\sum_{n=1}^{\infty}\overline{\theta}_{n}(Q(x))g_n(x).
    $$
Since the functions $\overline{\theta}_n (Q(x))$ form a partition
of unity subordinated to the covering $\bigcup_{n\in \N}C_n=X$,
$|g_n -f_n|\leq 1$, and $f_n=f$ on $C_n$, it is immediately
checked that
    $$
|\overline{g}(x)-f(x)|\leq 1 \textrm{ for all } x\in X.
    $$
On the other hand we have
\begin{eqnarray*}
& & |g(x)-\overline{g}(x)|=|\sum_{n=1}^{\infty}(\theta_n (Q(x))
-\overline{\theta}_n (Q(x)))g_n (x)|\\
& &\leq \sum_{n=1}^{\infty}\|g_{n}\|_{\infty} |\theta_n(Q(x))-
\overline{\theta}_n (Q(x))|
\\
& &\leq\sum_{n=1}^{\infty}\frac{1}{2^n}=1.
\end{eqnarray*}
By combining the last two estimations we get
    $$
|g(x)-f(x)|\leq 2 \textrm{ for all } x\in X.
    $$

In order to estimate $\textrm{Lip}(g)$, observe first that,
because support$(\overline{\theta}_n)=C_n$ and $f_n = f$ on $C_n$,
we can write, for every $x, y\in X$,
    $$
f(x)=\sum_{n=1}^{\infty}f(x)\overline{\theta}_n
(Q(y))=\sum_{n=1}^{\infty}f_n(x)\overline{\theta}_n (Q(y)).
    $$
On the one hand, for every $x, y\in X$,
\begin{eqnarray*}
& &
\overline{g}(x)-\overline{g}(y)=\sum_{n=1}^{\infty}\overline{\theta}_n
(Q(x))g_n (x) -\sum_{n=1}^{\infty}\overline{\theta}_n (Q(y))g_n
(y)=\\
& &\sum_{n=1}^{\infty}\left( g_n(x)- f_n(x)\right)
\left(\overline{\theta}_n (Q(x))- \overline{\theta}_n
(Q(y))\right) +
\sum_{n=1}^{\infty}(g_n(x)-g_n(y))\overline{\theta}_n (Q(y))\\
& &\leq \sum_{n=1}^{\infty}
\textrm{Lip}(\overline{\theta}_n)\textrm{Lip}(Q)\|x-y\|+\sum_{n=1}^{\infty}
C\|x-y\|\overline{\theta}_n (Q(y))\\
& &\leq (3\varepsilon\textrm{Lip}(Q)+C)\|x-y\|,
\end{eqnarray*}
so $\overline{g}$ is $(C+3\varepsilon\textrm{Lip}(Q))$-Lipschitz.
And on the other hand,
\begin{eqnarray*}
& & \textrm{Lip}(g-\overline{g})\leq
\sum_{n=1}^{\infty}\textrm{Lip}(Q)\textrm{Lip}(\theta_n
-\overline{\theta}_n)\|g_n\|_{\infty}+\sum_{n=1}^{\infty}\|\theta_n
-\overline{\theta}_n\|_{\infty}2\\
& &\leq \frac{\varepsilon}{2}+\frac{\varepsilon}{2}=\varepsilon,
\end{eqnarray*}
which implies
$$
\textrm{Lip}(g)\leq\textrm{Lip}(\overline{g})+\varepsilon\leq
C+4\varepsilon\textrm{Lip}(Q).
$$
Up to a replacement of $4\varepsilon\textrm{Lip}(Q)$ with
$\varepsilon$, we have thus shown that for every $\varepsilon>0$
and every $1$-Lipschitz function $f:X\to\R$ there exists a
$(C+\varepsilon)$-Lipschitz, real analytic function $g:X\to\R$
such that $|f-g|\leq 2$.

Finally, for arbitrary $\varepsilon \in (0, 2)$ and
$\textrm{Lip}(f):=L\in (0, \infty)$, we consider the function
$F:X\to\R$ defined by
$F(x)=\frac{2}{\varepsilon}f(\frac{\varepsilon}{2L} x)$, which is
$1$-Lipschitz so, by what we have just proved there exists a
$(C+\varepsilon)$-Lipschitz, real analytic function $G:X\to\R$
such that $|F-G|\leq 2$. If we
 define $g(x)=\frac{\varepsilon}{2} G(\frac{2L
}{\varepsilon}x)$, we get a real analytic function $g:X\to\R$ with
$\textrm{Lip}(g) \leq (C+\varepsilon)\textrm{Lip}(f)$, and such
that $|g-f|\leq\varepsilon$. This concludes the proof of Theorem
\ref{Lip approximation on Hilbert space}. $\square$

\begin{remark}
{\em However, note that the neighborhood of $X$ in $\widetilde{X}$
on which a holomorphic extension $\widetilde{g}$ of $g$ is defined
is $\widetilde{U}_{L,
\varepsilon}:=\frac{\varepsilon}{2L}\widetilde{U}$. In particular,
as $L=\textrm{Lip}(f)$ increases to $\infty$ (or as $\varepsilon$
decreases to $0$), the neighborhood $\widetilde{U}_{L,
\varepsilon}$ of $X$ where a holomorphic extension of the
approximation $g$ of $f$ is defined shrinks to $X$.

This unfortunate dependence of $\widetilde{U}_{L, \varepsilon}$ on
$L, \varepsilon$ (which, according to the Cauchy-Riemann
equations, is an inherent disadvantage of any conceivable scaling
procedure, not just the one we have used) prevents our going
further and combining our main result with a real analytic
refinement of Moulis' techniques \cite{M} in order to prove the
following conjecture: that every $C^1$ function defined on a
Banach space having a separating polynomial can be $C^1$-finely
approximated by real analytic functions.

It is no use, either, trying to avoid scaling by choosing $r$ of
the order of $\varepsilon/L$ and reworking the statements and the
proofs of the results following Lemma \ref{construction of the sup
partition of unity} in order to allow $r\in (0,1)$, because in
such case the neighborhood $\widetilde{V}$ where the holomorphic
extensions $\widetilde{\varphi}_n$ of the functions $\varphi_n$
are defined and conveniently bounded will be contained in
$\widetilde{V}_r=\{x+iy
: x, y\in X, \|y\|<r\}$, which also shrinks to $X$ as $r$ decreases to $0$.}
\end{remark}
Despite this difficulty, the real analytic sup partitions of unity
constructed in Lemma \ref{construction of the sup partition of
unity} can be used, together with Theorem \ref{analytic Lip
approximation}, to prove Theorem \ref{derivative approximation},
see \cite{AFK3}. In fact, an analysis of the proof of \cite{AFK3}
(in view of Proposition \ref{refinement of CHP} and following the
lines of Lemma \ref{holomorphic extension of the composition of
lambda with Phi} above) shows that the domain where a holomorphic
extension of the function $g$ in Theorem \ref{derivative
approximation} is defined and $\varepsilon$-close to $g$ only
depends on $\|f\|_\infty$, on $\varepsilon$, and on the modulus of
continuity of $f'$. Namely, we have the following sharp version of
Theorem \ref{derivative approximation}.

\begin{theorem}\label{sharp version of derivative approximation}
Let $X$ be a separable Banach space with a separating polynomial,
and let $M, K, \varepsilon>0$ be given. Then there exists an open
neighborhood $\widetilde{U}$ of $X$ in $\widetilde{X}$, depending
only on $M, K, \varepsilon$, such that for every function $f\in
C^{1,1}(X)$ with $\|f\|_\infty\leq M$ and $\textrm{Lip}(f')\leq K$
there exists a real analytic function $g:X\to\mathbb{R}$, with
holomorphic extension $\widetilde {g}:\widetilde{U}\to\mathbb{C}$,
such that
\begin{enumerate}
\item $|f(x)-g(x)|\leq \varepsilon$ for all $x\in X$.
\item $|f'(x)-g'(x)|\leq\varepsilon$ for all $x\in X$.
\item $|\widetilde{g}(x+iy)-g(x)|\leq \varepsilon$ for all $z=x+iy\in\widetilde{U}$.
\end{enumerate}
\end{theorem}

\section{Proof of Theorem \ref{Lip approximation on Hilbert
space}}

Finally, by combining Theorem \ref{sharp version of derivative
approximation} with the Lasry-Lions sup-inf convolution
regularization technique \cite{LaLi} and with the above gluing
tube function and crown arguments, we will prove Theorem \ref{Lip
approximation on Hilbert space}.

We start considering the special case when $f$ is bounded.

\subsection{Proof of Theorem \ref{Lip approximation on Hilbert
space} in the case of a function $f:X\to [0,M]$.}

Given an $L$-Lipschitz function $f:X\to [0,M]$ defined on a
separable Hilbert space $X$, set
    $$
f_\lambda (x)=\inf\{ f(u) +\frac{1}{2\lambda}|x-u|^2 \, : \, u\in
X\}=\inf\{ f(x-u) +\frac{1}{2\lambda}|u|^2 \, : \, u\in X\}
    $$
$$
f^\mu (x)=\sup\{ f(u) -\frac{1}{2\lambda}|x-u|^2 \, : \, u\in
X\}=\sup\{ f(x-u) -\frac{1}{2\lambda}|u|^2 \, : \, u\in X\}.
    $$
Since the supremum (and the infimum) of a family of $L$-Lipschitz
functions is $L$-Lipschitz, it is clear that $f_\lambda$ and
$f^\mu$ are $L$-Lipschitz.

Now, since $f$ is bounded and uniformly continuous, according to
\cite{LaLi}, the function
    $$
g_{\lambda, \mu}(x):=(f_\lambda)^{\mu}(x)=\sup_{z\in X}\inf_{y\in
X}\{f(y)+\frac{1}{2\lambda}|z-y|^{2}-\frac{1}{2\mu}|x-z|^{2}\}
    $$
is well defined and has a Lipschitz derivative on $X$ satisfying
    $$
\textrm{Lip}(g_{\lambda, \mu}')\leq\max\{\frac{1}{\mu},
\frac{1}{\lambda-\mu}\},
    $$
for all $0<\mu<\lambda$ small enough, and converges to $f(x)$,
uniformly on $X$, as $0<\mu<\lambda\to 0$. In fact, as noted in
\cite{LaLi}, the rate of convergence of $g_{\lambda, \mu}$ to $f$
only depends on $\textrm{Lip}(f)$, so for every $\varepsilon>0$
there exists $\lambda_0>0$ (only depending on $\varepsilon$ and
$L$) so that $|g_{\lambda, \mu}(x)-f(x)|\leq\varepsilon/2$ for all
$x\in X$, $0<\mu<\lambda\leq\lambda_0$.

Also, according to the above observations, this function is
$L$-Lipschitz. Therefore we have
    $$
\|g_{\lambda, \mu}'(x)\|\leq L, \textrm{ and } |f(x)-g_{\lambda,
\mu}(x)|\leq\frac{\varepsilon}{2}
    $$
for all $x\in X$, for some $0<\mu<\lambda$ small enough. Now fix
$\lambda,\mu$ with
    $$
0<\lambda<\lambda_0, \,\,\, \mu:=\frac{\lambda}{2},
    $$
and apply Theorem \ref{derivative approximation} to obtain a real
analytic function $g:X\to\mathbb{R}$ such that
$$|g_{\lambda, \mu}(x)-g(x)|\leq\frac{\varepsilon}{2} \textrm{ and
} |g_{\lambda, \mu}'(x)-g'(x)|\leq\frac{\varepsilon}{2}$$ for all
$x\in X$. By combining the last two inequalities we get
$|f(x)-g(x)|\leq\varepsilon$ for all $x\in X$ and
$\textrm{Lip}(g)\leq L+\varepsilon$. Moreover $g_{\lambda, \mu}$
has a holomorphic extension to a neighborhood $\widetilde{U}$ of
$X$ in $\widetilde{X}$ which only depends on $L$, on
$\textrm{Lip}(g_{\lambda, \mu}')$, on $M$, and on $\varepsilon$.
Since $\textrm{Lip}(g_{\lambda, \mu}')\leq \max\{1/\mu,
1/(\lambda-\mu)\}=2/\lambda$ and in turn $\lambda$ only depends on
$\varepsilon$ and on $\textrm{Lip}(f)\leq L$, we have thus proved
the following.

\begin{proposition}\label{almost preserving Lip approx in case of a function taking values in 01}
Let $X$ be a separable Hilbert space. For every $L, M,
\varepsilon>0$ there exists a neighborhood
$\widetilde{U}:=\widetilde{U}_{L, M, \varepsilon}$ of $X$ in
$\widetilde{X}$ such that, for every $L$-Lipschitz function
$f:X\to [0,M]$ there exists a real analytic function
$g:X\to\mathbb{R}$, with holomorphic extension $\widetilde
{g}:\widetilde{U}\to\mathbb{C}$, such that
\begin{enumerate}
\item $|f(x)-g(x)|\leq \varepsilon$ for all $x\in X$.
\item $g$ is $(L+\varepsilon)$-Lipschitz.
\item $|\widetilde{g}(x+iy)-g(x)|\leq \varepsilon$ for all $z=x+iy\in\widetilde{U}$.
\end{enumerate}
\end{proposition}
In particular this establishes Theorem \ref{Lip approximation on
Hilbert space} in the case when $f$ is bounded.

Now, fix $L=M=1$ and $\varepsilon>0$, and let $r, \delta\in (0,
\varepsilon/64)$ be as in Lemma \ref{tube}. Given a bounded
$1$-Lipschitz function $f:X\to\R$, we may assume that $f$ takes
values in the interval $[0, +\infty)$. Let $N\in\N$ be such that
$N\geq f(x)\geq 0$ for all $x\in X$. By Proposition \ref{almost
preserving Lip approx in case of a function taking values in 01},
we can find real analytic functions $g_1, ..., g_N$, with
holomorphic extensions $\widetilde{g}_1, ... \widetilde{g}_N$
defined on $\widetilde{U}:=\widetilde{U}_{1, 1, \min\{r/2,
\delta\}}\supset X$, such that
\begin{enumerate}
\item $|f_i(x)-g_i(x)|< r/2$ for all $x\in X$.
\item $g_i$ is $(1+\varepsilon)$-Lipschitz.
\item $|\widetilde{g_i}(x+iy)-g_i(x)|\leq \delta$ for all $z=x+iy\in\widetilde{U}$.
\end{enumerate}

Let us define $g=e_{1}^{*}\circ G^{-1}
\circ\{g_i\}_{i=1}^{\infty}=H\circ\{g_i\}_{i=1}^{\infty}$, where
we understand $g_i=0$ for all $i>N$.

Since $\textrm{Lip}(g_i)\leq 1+\varepsilon$ for all $i$, and $DH$
is bounded by $1+\varepsilon$, it follows that
$g=H\circ\{g_i\}_{i=1}^{\infty}$ is $(1+\varepsilon)^2$-Lipschitz.

Because $\left(  f_{1}\left(  x\right)  ,\cdots, f_{N}\left(
x\right), 0, 0, \cdots \right)  \in\beta\left(  \left[  0,N\right]
\right) ,$ and also $\left\vert f_{i}-g_{i}\right\vert <r/2,$ we
have $\left( g_{1}\left(  x\right) ,\cdots, g_{N}\left(  x\right),
0, 0, \cdots \right)  \in \mathcal{T}$ and, by property $(2)$ of
the preceding Lemma,
    $$
|g(x)-f(x)|=|H(\{g_n(x)\}_{n=1}^{\infty})-h(\{f_n(x)\}_{n=1}^{\infty})|\leq
\varepsilon.
    $$
The function $g$ is clearly real analytic, with holomorphic
extension $\widetilde{g}(z)=(\widetilde{e}_{1}^{*}\circ
\widetilde{G^{-1}})(\widetilde{g}_{1}(z), \widetilde{g}_{2}(z),
..., \widetilde{g}_{N}(z), 0, 0, ...)$ defined on $\widetilde{U}$.
And, because $|\widetilde{g_i}(x+iy)-g_i(x)|\leq \delta$ for all
$z=x+iy\in\widetilde{U}$, we have, using property $(5)$ of the
Lemma, that $|\widetilde{g}(x+iy)|\leq 2(|g(x)|+ 1)$ for all
$x+iy\in\widetilde{U}$.

Up to a change of $\varepsilon$ we have thus proved the following.

\begin{proposition}
Let $X$ be a separable Hilbert space. For every $\varepsilon>0$
there exists an open neighborhood
$\widetilde{U}=\widetilde{U}_{\varepsilon}$ of $X$ in
$\widetilde{X}$ such that, for every $1$-Lipschitz, bounded
function $f:X\to\R$, there exists a real analytic function
$g:X\to\mathbb{R}$, with holomorphic extension $\widetilde
{g}:\widetilde{U}\to\mathbb{C}$, such that
\begin{enumerate}
\item $|f(x)-g(x)|\leq \varepsilon$ for all $x\in X$.
\item $g$ is $(1+\varepsilon)$-Lipschitz.
\item $|\widetilde{g}(x+iy)|\leq 2(|g(x)|+ 1)$ for all $z=x+iy\in\widetilde{U}$.
\end{enumerate}
\end{proposition}

Then, by using exactly the same argument as in section 5.6 above
(just noting that now we are lucky to have $C=1+\varepsilon$), one
immediately obtains Theorem \ref{Lip approximation on Hilbert
space} in full generality.

\bigskip


\end{document}